\documentclass[a4paper,11pt]{article}

\usepackage{amsfonts,qtree,stmaryrd,textcomp}
\usepackage{pifont,amssymb,amsmath,amsthm,tabularx,graphicx}
\usepackage{tree-dvips,pictexwd,dcpic}
\usepackage{bbm}
\usepackage{bm}

\usepackage{stmaryrd}

\usepackage[T1]{fontenc}
\usepackage[latin1]{inputenc}
\usepackage[text={17cm,26cm},centering]{geometry}

\usepackage{etex}
\usepackage{fancyhdr}
\usepackage{graphicx}
\usepackage{enumitem}
\usepackage{amsmath}
\usepackage[bbgreekl]{mathbbol}
\usepackage{amssymb}
\usepackage{amsthm}
\usepackage[all]{xy}

\usepackage{epigraph}
\RequirePackage{bussproofs}

\usepackage{cancel}

\usepackage{mathabx}

\usepackage{scalerel,stackengine}

\usepackage{framed}
\usepackage{mathtools}
\usepackage{url}
\usepackage{proof}
\usepackage{xspace}
\usepackage{listings}
\usepackage{wrapfig}
\usepackage{tikz}
\usepackage{tikz-cd}

\usepackage{relsize}

\usetikzlibrary{positioning}
\usetikzlibrary{shapes}

\usetikzlibrary{arrows}
\usetikzlibrary{calc}
\usepackage{tikz-3dplot}
\usetikzlibrary{decorations.pathmorphing}

\usepackage{amsfonts,amssymb,amsmath,qtree,stmaryrd,textcomp, amsthm}

\usetikzlibrary{decorations.pathreplacing}
\tikzset{axis/.style={&lt;-&gt;}}

\stackMath
\newcommand\reallywidehat[1]{%
\savestack{\tmpbox}{\stretchto{%
  \scaleto{%
    \scalerel*[\widthof{\ensuremath{#1}}]{\kern-.6pt\bigwedge\kern-.6pt}%
    {\rule[-\textheight/2]{1ex}{\textheight}}
  }{\textheight}%
}{0.5ex}}%
\stackon[1pt]{#1}{\tmpbox}%
}
 \definecolor{MyBlue}{rgb}{0.05, 0.25, 0.65}
 
 \definecolor{MyRed}{rgb}{0.90, 0.05, 0.05}
 
\definecolor{MyGreen}{rgb}{0.05, 0.90, 0.05}

\newcommand{\B}{\boldsymbol}
\newcommand{\C}[1]{\mathcal{#1}}
\newcommand{\D}[1]{\mathbb{#1}}

\usepackage{ mathrsfs }

\newtheorem{theorem}{Theorem}[section]
\newtheorem{proposition}[theorem]{Proposition}
\newtheorem{lemma}[theorem]{Lemma}

\newtheorem{remark}[theorem]{Remark}

\newtheorem{example}[theorem]{Example}

\newtheorem{definition}[theorem]{Definition}

\newcommand{\Nat}{{\mathbb N}}
\newcommand{\Real}{{\mathbb R}}

\newcommand{\id}{\mathrm{id}}

\newcommand{\BISH}{\mathrm{BISH}}

\newcommand{\CST}{\mathrm{CST}}

\newcommand{\dom}{\mathrm{dom}}

\newcommand{\PEM}{\mathrm{PEM}}

\newcommand{\Top}{\mathrm{\mathbf{Top}}}

\newcommand{\INT}{\mathrm{INT}}
\newcommand{\CLASS}{\mathrm{CLASS}}

\newcommand{\TOT}{\Leftrightarrow}

\newcommand{\To}{\Rightarrow}

\newcommand{\sto}{\rightsquigarrow}

 \newcommand{\pto}{\rightharpoonup}

\newcommand{\MLTT}{\mathrm{MLTT}} 
\newcommand{\CZF}{\mathrm{CZF}}

\newcommand{\pr}{\textnormal{\texttt{pr}}}

\newcommand{\BST}{\mathrm{BST}}

\newcommand{\Disj}{\mathbin{\B ) \B (}}

\newcommand{\Set}{\mathrm{\mathbf{Set}}}

\newcommand{\se}{\mathrm{se}}

\newcommand{\eto}{\hookrightarrow}

\newcommand{\Fun}{\textnormal{\textbf{Fun}}}

\newcommand{\crTop}{\textnormal{\textbf{crTop}}}

\newcommand{\emptys}{\cancel{\mathlarger{\mathlarger{\mathlarger{\square}}}}}

\newcommand{\MIN}{\mathrm{MIN}}

\newcommand{\Ineq}{\textnormal{\texttt{Ineq}}}

\newcommand{\SetIneq}{\textnormal{\textbf{SetIneq}}}

\newcommand{\SetExtIneq}{\textnormal{\textbf{SetExtIneq}}}

\newcommand{\fXY}{f \colon X \to Y}

\newcommand{\swap}{\textnormal{\texttt{swap}}}

\newcommand{\Eisj}{\B ] \B [}

\newcommand{\swapa}{\textnormal{\texttt{swapa}}}

\newcommand{\ti}{\mathrm{I}}

\newcommand{\tii}{\mathrm{II}}

\newcommand{\SwapAlg}{\textnormal{\textbf{SwapAlg}}}

\newcommand{\SepSwapAlg}{\textnormal{\textbf{SepSwapAlg}}}

\newcommand{\lcm}{\textnormal{\texttt{lcm}}}

\newcommand{\gcdi}{\textnormal{\texttt{gcd}}}

\newcommand{\DPSA}{\mathrm{DPSA}}

\newcommand{\one}{\D 1}

\newcommand{\two}{\D 2}

\newcommand{\EEmpty}{\mathrm{Empty}} 

\newcommand{\coEmpty}{\mathrm{coEmpty}}

\newcommand{\StrExtFun}{\textnormal{\textbf{StrExtFun}}}

\newcommand{\Emb}{\textnormal{\textbf{Emb}}}

\newcommand{\Cont}{\textnormal{\textbf{Cont}}}

\newcommand{\Stone}{\textnormal{\texttt{Stone}}}

\newcommand{\fii}{\textnormal{\texttt{field}}}

\newcommand{\ho}{\textnormal{\texttt{char}}}

\newcommand{\PartSwapHom}{\textnormal{\textbf{PartSwapHom}}}

\newcommand{\Tot}{\textnormal{Tot}}

\newcommand{\EFQ}{\textnormal{EFQ}}

\newcommand{\bool}{\textnormal{\texttt{bool}}}

\begin{document}

\date{}

\title{\textbf{Constructive Stone representations for separated swap and Boolean algebras}}

\author{Daniel Misselbeck-Wessel\\	
	Munich Center for Mathematical Philosophy, University of Munich\\
	daniel.wessel@lmu.de\\[2mm]
	Iosif Petrakis\\
	Mathematics Institute, University of Munich\\
	petrakis@math.lmu.de} 
	

%





\maketitle

\begin{abstract}
	\noindent 
	Swap algebras generalise Bishop's complemented powerset as Boolean algebras generalise the powerset. Actually, all Boolean algebras are swap algebras. We prove constructively a Stone representation theorem for separated swap algebras of type $(\tii)$, where the notion of a separated swap algebra generalises the corresponding notion of a separated Boolean algebra.  Moreover, we prove a Stone-\v Cech theorem for swap algebras of type $(\tii)$, showing that the restriction to separated swap algebras is not a loss of generality from the point of view of the theory of swap characters. A constructive Stone representation theorem and a Stone-\v Cech theorem for Boolean algebras follow as special cases. We introduce sets with a Boolean inequality, that is sets with an internal falsum. These sets allow a book-keeping of the use of the Ex falso principle in constructive mathematics. If we restrict to swap algebras with a Boolean inequality, then the proof of the Stone representation theorem for swap algebras of type $(\tii)$ is within minimal logic.\\[2mm]
\end{abstract}


\section{Two problems in the constructivisation of the Stone representation theorem for Boolean algebras}\label{subsec: two}

The theory of swap algebras and swap rings is a generalisation of the theory of Boolean algebras and Boolean rings that originates from 
Bishop-style constructive mathematics $(\BISH)$ (see~\cite{Bi67, BB85, BR87}) and has
 applications also to classical mathematics $(\CLASS)$ (see~\cite{MWP24, MWP24b}). Swap algebras and swap rings are introduced in~\cite{MWP24} as abstract versions of the class of complemented subsets and partial Boolean-valued functions, respectively, that were 
 studied systematically in~\cite{PW22}. Complemented subsets  i.e., pairs of subsets which are disjoint in a strong sense, were introduced by Bishop in~\cite{Bi67} as an important tool to his constructive reconstruction of measure and integration theory\footnote{In~\cite{Bi66}, p. 309, Bishop forsees that the use of complemented subsets in mathematics can be as broad as the use of complementation itself:
 \begin{quote}
 	We do not wish to define $x \in - A$ to mean that the assumption $x \in A$ leads to a contradiction (in contrast to the approach of Brouwer). Complementation defined in terms of negation is too elusive. In addition, it leads to a loss of meaning. Therefore whenever a notion of set complementation is needed we introduce it affirmatively. (The same  is true for inequality relations) One way of doing this is through the very flexible notion of a complemented set.
 	\end{quote}}. As the powerset $\C P(X)$ of a set $X$ is classically  bijective to the 
 \textit{total} Boolean-valued functions on $X$, the complemented power set $\C P^{\Disj}(X)$ of $X$ is constructively bijective to the \textit{partial} Boolean-valued functions on $X$ (see~\cite{PW22, MWP24}). A Boolean algebra and a Boolean ring is a special case of a swap algebra and a swap ring, respectively. In~\cite{MWP24} it is shown that the duality between swap algebras of type\footnote{There are two types of swap algebras, as there are two algebras of complemented subsets i.e., two ways to define their join and meet (see~\cite{PW22, MWP24}).} $(\tii)$ and swap rings generalises the duality between Boolean algebras and Boolean rings.
Here we extend our results in~\cite{MWP24} presenting a constructive Stone representation theorem for separated swap algebras of type $(\tii)$ that has as a special case a constructive Stone representation theorem for separated Boolean algebras.

We highlight two problems in the constructivisation of the Stone representation theorem for Boolean algebras. \textbf{The first problem} can be called ``the use of points''. For example, if one wants to define pointed Boolean-valued homomorphisms, or pointed Boolean characters\footnote{The term character comes from the theory of Banach algebras. The classical Stone representation theorem is a special case of the Gelfand-Naimark theorem for commutative $C^*$-algebras as the notion of a $\D C$-valued character on a $C^*$-algebra is reduced to a Boolean-valued homomorphism on a Boolean algebra (see~\cite{BD73, DB86}).} on $\C P(X)$, then one needs to use the principle of the excluded middle $(\PEM)$; if $x_0 \in X$, then $\widehat{x_0} \colon \C P(X) \to \D 2$, where $\D 2 := \{0, 1\}$, is defined with $\PEM$ by the rule
$$\widehat{x_0}(A) := \left\{ \begin{array}{ll}
	1   &\mbox{, $x_0 \in A$}\\
	0             &\mbox{, $x_0 \notin A$.}
\end{array}
\right. $$
It is with these pointed Boolean characters that one shows that $\C P(X)$ is separated. We call a Boolean algebra \textit{separated}, if 
for every $a \neq 0$ there is a Boolean character, $h_a$ on $A$ with $h_a(a) = 1$. If $A \neq \emptyset$, then classically there is a point $x_0 \in A$, hence $\widehat{x_0}(A) = 1$.  While point-free approaches to the theory of Boolean algebras provide a way out to this dead end, from the constructive point of view (see e.g.,~\cite{Ne96},~\cite{CS08}), here we elaborate a different constructive ``solution'' using points, but within the ``two-dimensional'' framework of complemented subsets.

Following the successful use of complemented subsets and real-valued, partial functions in constructive measure theory by Bishop, Cheng and Bridges in~\cite{BC72, Br79, BB85}, and later by other authors (see~\cite{CP02, Ch21, GP23, PZ22}), we employ here complemented subsets and \textit{swap characters} i.e., partial, Boolean-valued homomorphisms, in order to avoid the aforementioned use of PEM in the definition of pointed Boolean characters.
If $\B A := (A^1, A^0)$ is a complemented subset of set $X$ with an extensional inequality, where $A^1, A^0$ are disjoint in a strong sense (see Definition~\ref{def: extapartsubsets}), then we define the above pointed Boolean character, not on the whole complemented powerset, but on its subset of all complemented subsets $\B A$ for which their domain $A^1 \cup A^0$ contains $x_0$, using the following rule that avoids $\PEM$: 
$$\widehat{x_0}(\B A) := \left\{ \begin{array}{ll}
	1   &\mbox{, $x_0 \in A^1$}\\
	0             &\mbox{, $x_0 \in A^0$.}
\end{array}
\right. $$
With these pointed swap characters we show constructively (Proposition~\ref{prp: sep1}) that $\C P^{\Disj}(X)$ is a separated swap algebra of type $(\tii)$, a notion that generalises the notion of a separated Boolean algebra and involves the set $\widehat{\B A}$ of swap characters on a swap algebra $A$ of type $(\tii)$ (Definition~\ref{def: sep}). 

\textbf{The second problem} in the constructivisation of the Stone representation theorem is that one needs Zorn's lemma to show that every Boolean algebra is separated
(see~\cite{Ha74}, p.~77). Actually, the Boolean Prime Ideal Theorem, a non-constructive choice principle, suffices (see~\cite{BH21}). Here we explain that there is good reason to avoid the classical fact that every Boolean algebra is separated. We do not need to employ some maximal object, only to change the equality of $A$, in order to restrict our study to separated Boolean algebras only. As we show in section~\ref{sec: stonecech}, from a certain point of view this move
does not imply a loss of generality. In topology a similar attitude is followed in the theory of the ring of continuous functions $C(X)$ of a topological space $X$ (see~\cite{GJ60, Wa74}), where from the point of view of $C(X)$ it suffices to restrict to completely regular topological spaces.
Similarly, from the point of view of the theory of swap characters $\widehat{\B A}$, 
it suffices to restrict to separated swap (Boolean) algebras.

We structure this paper as follows:
\vspace{-2mm}
\begin{itemize}
	\item In section~\ref{sec: extsub} we give some basic definitions and facts on sets equipped with an extensional inequality ($\SetExtIneq$), on extensional subsets of such sets, and on strongly extensional functions ($\StrExtFun$). Our main ``universe'' is the category\footnote{For clarity, we present a category as the pair of its objects and its arrows.} $\big(\SetExtIneq, \StrExtFun\big)$
	of sets with an extensional inequality and strongly extensional functions.

	\item In section~\ref{sec: cs} we present some basic definitions and facts on complemented subsets.

	\item In section~\ref{sec: bvpf} we introduce a new equality and inequality on partial Boolean-valued functions on a set $X$ in $\SetExtIneq$, that can be ``inherited'' to complemented subsets (Definition~\ref{def: csineq}).

	\item In section~\ref{sec: gelfand} we define the partial Gelfand transform as a constructive counterpart to the classical total Gelfand transform, and we show that 
	the partial Gelfand transform is a well-defined embedding of a set $X$ in $\SetExtIneq$ into its second strongly extensional, partial dual $\B X^{\circledast \circledast}$ (Proposition~\ref{prp: pGelfand}).

	\item In section~\ref{sec: swap} we include the definition of swap algebras of type $(\tii)$, some basic examples of such algebras, and a few of their fundamental properties. Using intuitionisitic logic $(\INT)$, the complemented powerset is a swap algebra of type $(\tii)$.

	\item In section~\ref{sec: shomos} we define the set $\widehat{\B A}$ of swap characters 
	(Definition~\ref{def: partialshomo}) between a swap algebra of type $(\tii)$ and the Booleans $\D 2$. We introduce the notion of a separated swap algebra (Definition~\ref{def: sep}), and we show that the complemented subsets of a set $X$ in $\SetExtIneq$ form a separated swap algebra of type $(\tii)$ (Proposition~\ref{prp: sep1}). The category $(\SwapAlg_{\tii}, \PartSwapHom)$ of swap algebras of type $(\tii)$ and partial swap-homomorphisms between them is introduced in Proposition~\ref{prp: compswap}.

	\item In section~\ref{sec: bool} we introduce sets with a Boolean inequality. These sets facilitate the avoidance of the Ex falso principle $(\EFQ)$ in many proofs of constructive mathematics. We introduce the so-called Boolean complemented subsets, and we show that within minimal logic $(\MIN)$ the Boolean complemented powerset is a swap algebra of type $(\tii)$ (Proposition~\ref{prp: bcs1}) under the condition that the Boolean inequality is complete. The latter property (Definition~\ref{def: complete}) is satisfied by all fundamental number systems (Proposition~\ref{prp: NtoR}).

	\item In section~\ref{sec: repr} we prove constructively a Stone representation theorem for separated swap algebras of type $(\tii)$ (Theorem~\ref{thm: stone}), according to which a separated swap algebra $A$ of type $(\tii)$ is embedded into the swap algebra $\C E^{\Disj}(\widehat{\B A})$ of type $(\tii)$ of complemented subsets of $\widehat{\B A}$. If $A$ is a set with  a Boolean inequality, then the proof of our Stone representation theorem is within minimal logic (Proposition~\ref{prp: min}).

	\item In section~\ref{sec: stonecech} we prove a Stone-\v{C}ech theorem for swap algebras of type $(\tii)$ (Theorem~\ref{thm: sc}), according to which the swap characters on a swap algebra $A$ of type $(\tii)$ can be identified with the swap characters on a separated swap algebra $\sigma A$ of type $(\tii)$. Moreover, if $A$ has a Boolean inequality, then $\sigma A$ also has a Boolean inequality.

\item In section~\ref{sec: ba} we reduce the main results of the previous two sections to the special case of Boolean algebras. Namely, from  Theorem~\ref{thm: stone} follows a constructive Stone representation theorem for separated Boolean algebras (Theorem~\ref{thm: stoneba}), and from Theorem~\ref{thm: sc} follows a Stone-\v{C}ech theorem for Boolean algebras (Theorem~\ref{thm: scba}).

\end{itemize}

We work within Bishop Set Theory $(\BST)$, a semi-formal system for $\BISH$ that behaves as a high-level programming language. For all notions and results of $\BST$ that are used here without definition or proof we refer to~\cite{Pe20, Pe22a, Pe24}. The type-theoretic interpretation of Bishop's set theory into the theory of setoids (see especially the work of Palmgren~\cite{Pa12}--\cite{Pa19}) has become nowadays the standard way to understand Bishop sets. Other formal systems for BISH are Myhill's Constructive Set Theory $(\CST)$, introduced 
in~\cite{My75}, and Aczel's system $\CZF$ (see~\cite{AR10}).
For all notions and results from the theory of swap algebras that are used here without definition or proof, we refer to~\cite{PW22, MWP24}.

\section{Extensional subsets}
\label{sec: extsub}

Bishop Set Theory (BST) is an informal, constructive theory of \emph{totalities} and \emph{assignment routines} between totalities that accommodates Bishop's informal system BISH and serves as an intermediate step between Bishop's original theory of sets and an adequate and faithful formalisation of BISH in Feferman's sense~\cite{Fe79}. Totalities, apart from the basic, undefined set of natural numbers $\mathbb{N}$, are defined through a membership-condition. The \emph{universe} $\mathbb{V}_0$ of predicative sets is an open-ended totality, which is not considered a set itself, and every totality the membership-condition of which involves quantification over the universe is not considered a set, but a proper class. \emph{Sets} are totalities the membership-condition of which does not involve quantification over $\mathbb{V}_0$, and are equipped with an equality relation i.e., an equivalence relation. An equality relation $x =_X x{'}$ on a defined set $X$ is defined through a formula $E_X(x, x{'})$ and a proof that $E_X(x, x{'})$ satisfies the properties of an equivalence relation.
Assignment routines are of two kinds: \emph{non-dependent} ones and \emph{dependent} ones. 

\begin{definition}\label{def: function}
	If $(X, =_X)$ and $(Y, =_Y)$ are sets, a function $\fXY$ is a non-dependent assignment routine $f \colon X \sto Y$ i.e., for every $x \in X$ we have that $f(x) \in Y$, such that $x =_X x{'} \To f(x) =_Y f(x{'})$, for every $x, x{'} \in X$. A function \(f:X\to Y\) is an embedding, if for all \(x,x'\) in \(X\) we have that \(f(x) =_Yf(x')\) implies \(x =_X x'\), and we write \(f:X \eto Y\), while it is a surjection, if $\forall_{y \in Y}\exists_{x \in X}(f(x) =_Y y )$. We denote by $\D F(X, Y)$ the set of functions from $X$ to $Y$, which is equipped with the equality\footnote{In this way  the type-theoretic function extensionality axiom is built in as the canonical equality of the function set. }
	$$f =_{\D F(X,Y)} g :\TOT \forall_{x \in X}\big(f(x) =_Y g(x).$$
	Let $\id_X$ be the identity function on $X$. Let $(\Set, \Fun)$ be the category of sets and functions.
	Two sets are equal in $\mathbb{V}_0$ if there is $g \colon Y \to X$, such that\footnote{With this equality the universe in $\BST$ can be called \emph{univalent} in the sense of Homotopy Type Theory~\cite{Ri22, Ho13}, as, by definition, an equivalence between sets is an equality.} $g \circ f =_{\D F(X, X)} \id_X$ and $f\circ g =_{\D F(Y, Y)} \id_Y$.
\end{definition}

\begin{definition}\label{def: apartness}
	If $(X, =_X)$ is a set, let the following formulas with respect to a relation $x \neq_X y:$\\[1mm]
	$(\Ineq_1) \  \forall_{x, y \in X}\big(x =_X y \ \& \ x \neq_X y \To \bot \big)$, where $\bot := 0 =_{\Nat} 1$.\\[1mm]
	$(\Ineq_2) \ \forall_{x, x{'}, y, y{'} \in X}\big(x =_X x{'} \ \& \ 
	y =_X y{'} \ \& \ x \neq_X y \To x{'} \neq_X y{'}\big)$,\\[1mm]
	$(\Ineq_3) \ \forall_{x, y \in X}\big(\neg(x \neq_X y) \To x =_X y\big)$,\\[1mm]
	$(\Ineq_4)  \forall_{x, y \in X}\big(x \neq_X y \To y \neq_X x\big)$,\\[1mm]
	$(\Ineq_5) \ \forall_{x, y \in X}\big(x \neq_X y \To \forall_{z \in X}(z \neq_X x \ \vee \ z \neq_X y)\big)$,\\[1mm]
	$(\Ineq_6) \ \forall_{x, y \in X}\big(x =_X y \vee x \neq_X y\big)$.\\[1mm]
	If $\Ineq_1$ is satisfied, we call $\neq_X$ an 	inequality on $X$, and the 
	structure $\C X := (X, =_X, \neq_X)$ a set with an inequality. 
	If $(\Ineq_6)$ is satisfied, then $\C X$ is 
	\textit{discrete}, if $(\Ineq_2)$ holds, $\neq_X$ is 
	called \textit{extensional}, and  if $(\Ineq_3)$ holds, it is \textit{tight}. If it satisfies $(\Ineq_4)$ and $(\Ineq_5)$, it is called an \textit{apartness relation} on $X$.
\end{definition}

The primitive inequality $\neq_{\Nat}$ on $\Nat$ is a discrete, tight apartness relation.

\begin{definition}\label{def: setineq}
	If $\C X := (X, =_X, \neq_X)$ and $\C Y := (Y, =_Y, \neq_Y)$ are sets with inequality, a function $f \colon X \to Y$ is \textit{strongly extensional},
	if $f(x) \neq_Y f(x{'}) \To x \neq_X x{'}$, for every $x, x{'} \in X$, and it is an injection, if the converse implication $($also\footnote{Here we consider injections that are already strongly extensional functions.}$)$ holds i.e., $x \neq_X x{'} \To f(x) \neq_Y f(x{'})$, for every $x, x{'} \in X$. 
	Let $(\SetIneq, \StrExtFun)$ be the category of sets with an inequality and strongly extensional functions, and let	$(\SetExtIneq, \StrExtFun)$ be its subcategory with sets with an extensional inequality.	
\end{definition}

\begin{definition}\label{def: canonicalineq}
	Let $\C X := (X, =_X, \neq_X)$ and $\C Y := (Y, =_Y, \neq_Y)$ be sets with inequalities.
	The canonical inequalities on the product $X \times Y$ and the function space
	$\D F(X, Y)$ are given, respectively, by
	$$(x, y) \neq_{X \times Y} (x{'}, y{'}) :\TOT x \neq_X x{'} \vee y \neq_Y y{'},$$
	$$f \neq_{\D F(X, Y)} g :\TOT \exists_{x \in X} \big[f(x) \neq_Y g(x)\big].$$
\end{definition}
The projections $\pr_X, \pr_Y$ associated to $X \times Y$ 
are strongly extensional functions.
It is not easy to give examples of non strongly extensional functions, although we cannot accept
in $\BISH$ that all functions are strongly extensional. E.g., the strong extensionality of 
all functions from a metric space to itself is equivalent to Markov's principle (see~\cite{Di20}, p.~40).
Even to show that a constant function between sets with an inequality is strongly extensional, one needs 
intuitionistic,  and not minimal, logic. An apartness relation $\neq_X$ on a set $(X, =_X)$ is always
extensional (see~\cite{Pe20}, Remark 2.2.6, p.~11). Here we work only with extensional subsets, as in Bridges-Richman Set Theory i.e., the set-theoretic framework of~\cite{BR87, MRR88, BV06}. The use of extensional subsets only allows a smoother treatment of the (extensional) empty subset of a set, and hence a calculus of subsets closer to that of subsets within classical logic. The role of Myhill's axiom of non-choice to the relation between extensional and categorical subsets i.e., sets with an embedding to the given set, is studied in~\cite{MWP24}.

\begin{definition}[Extensional subsets]\label{def: extsubset}
	A formula $P(x)$, where $x$ is a variable of set $(X, =_X)$,
	is an \textit{extensional property}\footnote{Notice that in $\MLTT$ a property on a 
		type $X$ is described as a type family $P \colon X \to \C U$ over $X$, where $\C U$ is a universe and $X \colon \C U$, and it
		is always extensional through the corresponding transport-term (see~\cite{Ho13}, section 2.3).}
	on $X$\index{extensional property on a set}, if 
	$$\forall_{x, y \in X}\big([x =_{X } y \ \& \ P(x)] \To P(y)\big).$$
	The totality $X_P := \{x \in X \mid P(x)\}$ defined by the separation scheme for extensional propertes is called an \textit{extensional subset} of $X$, and its equality is inherited from $=_{X}$. If $\neq_X$ is a given inequality on $X$, then the canonical inequality
	$\neq_{X_P}$ on $X_P$ is inherited from $\neq_{X}$. The totality of extensional subsets of $X$ is denoted by $\C E(X)$.
	If $X_P, X_Q \in \C E(X)$, let $X_P \subseteq X_Q :\TOT \forall_{x \in X}\big(P(x) \To Q(x)\big)$ and $X_P =_{\C E(X)} X_Q :\TOT
	X_P \subseteq X_Q \ \& \ X_Q \subseteq X_P$. We may avoid writing the extensional property $P(x)$ determining an extensional
	subset $A$ of $X$, and in this case we simply write $A \subseteq X$. 
	\end{definition}

Notice that no quantification over the universe $\D V_0$ is required in the membership condition of 
$\C E(X)$. In accordance to the fundamental distinction between \textit{judgments} and \textit{propositions} due to Martin-L\"of, 
if $A, B \in \C E(X)$, the inclusion $A \subseteq B$ is a formula that does
not employ the membership symbol 
i.e., we do not say that for all $x \in X$ if $x \in A$, then $x \in B$, as such a definition would treat the judgments $x \in A$ and $x \in B$ as formulas. If $A := X_P$ and $B := X_Q$, we only need to prove $P(x) \To Q(x)$, for every $x \in X$. There are non-extensional properties on a set i.e., formulas $P(x)$, where $x$ is a variable of the set $X$, such that $P(x)$ and $x =_X y$ do not imply $P(y)$. 
For example, let $n \in \Nat$, $q \in \D Q$ and $P_q(x)$, where $x$ is a variable of set $\Real$, defined by $P_q(x) := x_n =_{\D Q} q.$ If $y \in \Real$ such that $y =_{\Real} x$, then it is not necessary
that $y_n =_{\D Q} q$, if $ x_n =_{\D Q} q$. A set $X$ is an extensional subset of itself, in the sense that it is equal in $\D V_0$ to the extensional subset $X_{=_{X}} := \{x \in X \mid x =_X x\}$.

\begin{definition}\label{def: complement}
The strong complement of a subset $A$ of $X \in \SetExtIneq$ is defined by
$$A^{\neq_X} := \{x \in X \mid \forall_{a \in A}(x \neq_X a)\}.$$
The weak complement of a subset $A$ of $X$ in $\Set$ is defined as the following extensional subset of $X\hspace{-1.5mm}:$
$$A^c := \big\{x \in X \mid \forall_{a \in A}\big(\neg (x =_X a)\big)\big\}.$$
\end{definition}

Sets $\one$ and $\two$ are defined as extensional subsets 
of the primitive set  $\Nat$:
$$\one := \{x \in \Nat \mid x =_{\Nat} 0\} =: \{0\}, \ \ \ \ \ 
\two := \{x \in \Nat \mid x =_{\Nat} 0 \ \vee x =_{\Nat} 1\} =: \{0, 1\}.$$

\begin{remark}\label{rem: boolese}
The standard operations $+, \wedge, \vee \colon \two \times \two \to \two$, and $\swap \colon \two \to \two$ are strongly extensional.
\end{remark}

\begin{proof}
We prove the strong extensionality of $+$ and for $\wedge$ and $\vee$ we work similarly. If $i, j, i{'}, j{'} \in \D2$, we suppose that $i + j \neq_{\D 2} i{'} + j{'}$, and we show that $i \neq_{\D 2} i{'}$ or $j \neq_{\D 2} j{'}$. Suppose that $i + j =_{\D 2} 1$ and $i{'} + j{'} =_{\D 2} 0$. From the first equality we have that $i =_{\D 2} 1$ and $j =_{\D 2} 0$ or $i =_{\D 2} 0$ and $j =_{\D 2} 1$, while from the second we get $i =_{\D 2} j =_{\D 2} = 0$ or $i =_{\D 2} j =_{\D 2} = 1$. In all cases the required disjunction follows trivially. The strong extensionality of $\swap$ is immediate to show.
\end{proof}

As an equality $=_X$ is extensional on $X \times X$, the diagonal $D(X, =_X) := \{(x, y) \in X \times X \mid x =_X y\}$ of $X$ is an extensional subset of $X \times X$. In the following definition of the extensional empty subset of a set $X$ we do not require $X$ to be inhabited, as Bishop does in~\cite{Bi67}, p.~65, for the categorical empty subset. We only require that $X$ is equipped with an extensional inequality. Notice that Bishop never defined the empty set, only the empty subset of an
inhabited set.

\begin{definition}[Extensional empty subsets]\label{def: extempty}
		If $\C X := (X, =_{X}, \neq_{X})$ in in $\SetExtIneq$, 
	the strong empty subset of $X$ is the extensional subset 
	$$\emptys_X := \{x \in X \mid x \neq_X x\}.$$
	We call	$X_P \in \C E(X)$ strongly empty, if $X_P \subseteq \emptys_X$.
	The weak empty subset of $X$ in $\Set$ is the extensional subset 
	$$\emptyset_X := \{x \in X \mid \neg(x =_X x)\}.$$
	We call	$X_P \in \C E(X)$ weakly empty, if $X_P \subseteq \emptyset_X$.
\end{definition}

Notice that within intuitionistic logic $(\INT)$ the inclusion $\ \emptys_X \subseteq X_P$ holds, for every extensional property $P(x)$ on $X$. The extensionality of an equality $=_X$ implies that the inequality $\neg(x =_X y)$ on $(X, =_X)$ is extensional and $\emptyset_X $ is an extensional subset of $X$.
As strong concepts suit better to constructive mathematics, we prefer to work with the strong variants of set-theoretic concepts, rather than the weak ones.
As the definition of the strong empty subset $\emptys_X$ requires that $X$ is equipped with an extensional inequality, we work within the category $(\SetExtIneq, \StrExtFun)$, although many related concepts are definable also within the category $(\SetIneq, \StrExtFun)$. 
For the definition of all the rest basic operations on extensional subsets we refer 
to~\cite{MWP24, Pe24a}.

\section{Complemented subsets}
\label{sec: cs}

Complemented subsets are pairs of strongly disjoint subsets of a set with an (extensional) inequality\footnote{Two subsets $A, B$ of a set $(X, =_X)$ \textit{overlap weakly}, if $A \cap B \neq \emptyset_X$ and they are \textit{weakly disjoint}, if $A \cap B = \emptyset_X$.}. 
	
	\begin{definition}\label{def: extapartsubsets}
		If $\C X := (X, =_X, \neq_X)$ is in $\SetExtIneq$ and 
		$A := X_P, B := X_Q \in \C E(X)$, then the $($strong$)$ overlapping relation $A \between B$ is defined by
		$$A \between B :\TOT \exists_{x \in X_P}\exists_{x{'} \in X_Q}\big(x =_X x{'}\big).$$
		$A$ and $B$ are $($strongly$)$ disjoint with respect to $\neq_X$, if
		$$A \Disj B :\TOT \forall_{x \in X_P}\forall_{x{'} \in X_Q}\big(x \neq_X x{'}\big).$$
		A \textit{complemented subset}\index{complemented subset} of $\C X$ is a pair
		$\B A := (A^1, A^0)$\index{$\B A := (A^1, A^0)$}, 
		where $A^1 := X_{P^1}, A^0 := X_{P^0} \in \C E(X)$ such that $A^1 \Disj A^0$. Its characteristic function $\chi_{\B A} \colon \dom(\B A) \to \two$, where $\dom(\B A) := A^1 \cup A^0$, is defined by 
		$$\chi_{\B A}(x) := \left\{ \begin{array}{ll}
			1   &\mbox{, $P^1(x)$}\\
			0             &\mbox{, $P^0(x)$.}
		\end{array}
		\right. $$
		We call a complemented subset $\B A$ of $X$:\\[1mm]
		\normalfont(i)
		\itshape inhabited, if $A^1$ is inhabited. \\[1mm]
		\normalfont(ii)
		\itshape coinhabited, if $A^0$ is inhabited.\\[1mm]
		\normalfont(iii)
		\itshape  total, if $\dom(\B A) =_{\C E(X)} X$.\\[1mm]
		\normalfont(iv)
		\itshape $1$-tight, if for every $x \in X$, we have that 
		$\neg{(x \in A^1)} \To x \in A^0$.\\[1mm]
		\normalfont(v)
		\itshape $0$-tight, if for every $x \in X$, we have that 
		$\neg{(x \in A^0)} \To x \in A^1$. 	\\[1mm]
		We denote by $\C E^{\Disj}(X)$\index{$\C P^{\Disj}(X)$} the totality of
		complemented extensional subsets of $X$.
		If $\B A, \B B \in \C E^{\Disj}(X)$, let
		$\B A \subseteq \B B : \TOT A^1 \subseteq B^1 \ \& \ B^0 \subseteq A^0$ and let
		$\B A =_{\C E^{\mathsmaller{\Disj}} (X)} \B B : \TOT \B A \subseteq \B B \ \& \ \B B \subseteq \B A$.  
	\end{definition}

	If the elements of $A^1$ are the ``provers'' of $\B A$ and the elements of $A^0$ are the ``refuters'' of $\B A$, then the inequality $\B A \subseteq \B B$ means that all provers of $\B A$ prove $\B B$ and all refuters of $\B B$ refute $\B A$ i.e., $\B B$ has more provers and less refuters than $\B A$.

	\begin{proposition}\label{prp: intcomp1}
	Let $\C X := (X, =_X, \neq_X)$ in $\SetExtIneq$ and $\B A := (A^1, A^0) \in \C E^{\Disj}(X)$.\\[1mm]
	\normalfont (i)
	\itshape $A^1 \cap A^0 \subseteq \emptys_X$.\\[1mm]
	\normalfont (ii) 
	\itshape $(\INT)$ $\emptys_X \subseteq A^1 \cap A^0$.
	\end{proposition}
	
	\begin{proof}
	(i) If $x \in X$ such that $P^1(x)$ and $P^0(x)$, where $P^1$ and $P^0$ are the extensional properties on $X$ that determine $A^1$ and $A^0$, respectively, then by the definition of $A^1 \Disj A^0$ we get $x \neq_X x$ i.e., $x \in \emptys_X$.\\
	(ii) It follows trivially by $\EFQ$.
	\end{proof}
	
	All operations between extensional complemented subsets
	are determined by the corresponding definitions of operations between extensional subsets.
	In this paper we shall use only the operations of the so-called in~\cite{PW22} \textit{second algebra of complemented subsets} $\C B_2^e(X) := \big(\C E^{\Disj}(X), \vee, \wedge, -)$ of $X$, as the first algebra $\C B_1^e(X) := \big(\C E^{\Disj}(X), \cup, \cap, -)$ does not suit well to swap rings (it suits well though, to constructive topologies of open complemented subsets (see~\cite{Pe24a})).

	\begin{definition}\label{def: opers}
		Let $\C X := (X, =_X, \neq_X), \C Y := (Y, =_Y, \neq_Y) \in \SetIneq$.
		If $\B A, \B B \in \C E^{\Disj}(X)$, let
		$$\B A \vee \B B  := \big((A^1 \cap B^1) \cup (A^1 \cap B^0) \cup (A^0 \cap B^1), A^0 \cap B^0\big),$$ 
		$$\B A \wedge \B B := \big((A^1 \cap B^1),  (A^1 \cap B^0) \cup (A^0 \cap B^1) \cup (A^0 \cap B^0)\big), $$
		$$ \ \ \ - \B A := (A^0, A^1), \ \ \ \ \ \ \B A - \B B := \B A \wedge (-\B B),$$
		\end{definition}

	With the above notion of complement of a complemented subset we recover constructively properties of classical subsets that are based on classical logic. Namely, we have that 
	$$\B A \subseteq \B B \TOT -\B B \subseteq - \B A, \ \ \ \ -(-\B A) := \B A, \ \ \ \ (X, \emptys_X) - \B A =_{\C E^{\Disj}(X)} -\B A.$$
   The alternative definitions of the join and the meet of complemented subsets in the first algebra $\C B_1^e(X) := \big(\C E^{\Disj}(X), \cup, \cap, -)$ are the following:
	$$\B A \cup \B B  := \big(A^1 \cup B^1, A^0 \cap B^0\big), \ \ \ \ \ \B A \cap \B B  := \big(A^1 \cap B^1, A^0 \cup B^0\big).$$
	In contrast to the standard one-dimensional subset theory, where only one empty subset $\emptys_X$ of a set $X$ exists, in the two-dimensional complemented subset theory there are many ``empty'' complemented subsets and equally many ``coempty'' subsets.

	\begin{definition}\label{def: zeros}
		If $\C X := (X, =_X, \neq_X)$ is in $\SetExtIneq$, and if $\B A \in \C E^{\Disj}(X)$, let 
		$$0_{\B X} := ( \emptys_X, X), \ \ \ \ 1_{\B X} := (X,  \emptys_X), \ \ \ \ 
		1_{\B A} :=  (\dom(\B A), \emptys_X), \ \ \ \  0_{\B A} := \big(\emptys_X, \dom(\B A)\big).$$
		The complemented subsets $0_{\B X} $ and of the form $0_{\B A}$ are called empty, while the 
		complemented subsets $1_{\B X} $ and of the form $1_{\B A}$ are called coempty. Let their corresponding sets
		$$\EEmpty(\C X) := \{\B A := (A^1, A^0) \in \C E^{\Disj}(X) \mid A^1 =_{\C E(X)} \emptys_X\},$$
		$$\coEmpty(\C X) := \{\B A := (A^1, A^0) \in \C E^{\Disj}(X) \mid A^0 =_{\C E(X)} \emptys_X\}.$$
	\end{definition}
	
	For the straightforward proof of the next proposition we repeatedly use $\EFQ$ to show the inclusion $\ \emptys_X \subseteq A$, where $A$ is an extensional subset of a set $X$ with an extensional inequality.
	
	\begin{proposition}[$\INT$]\label{prp: isswapa1} If $\C X := (X, =_X, \neq_X)$ is in $\SetExtIneq$ and $\B A \in \C E^{\Disj}(X)$, the following hold:\\[1mm]
		\normalfont(i)
		\itshape
		$0_{\B X}$ and $1_{\B X}$ are the bottom and top elements of $\C E^{\Disj}(X)$, respectively.\\[1mm]
		\normalfont(ii)
		\itshape
		$-0_{\B X} := 1_{\B X}$ and $-0_{\B A} := 1_{\B A}$.\\[1mm]
		\normalfont(iii)
		\itshape
		$\B A \vee (- \B A) =_{\mathsmaller{\C E^{\Disj}(X)}} 1_{\B A}$ and 
		$\B A \wedge (- \B A) =_{\mathsmaller{\C E^{\Disj}(X)}} 0_{\B A}$.\\[1mm]
		\normalfont(iv)
		\itshape
		$0_{\B X} \vee \B A =_{\mathsmaller{\C E^{\Disj}(X)}} \B A =_{\mathsmaller{\C E^{\Disj}(X)}} 
		0_{\B A} \vee \B A$.\\[1mm]
		\normalfont(v)
		\itshape
		$0_{\B X} \wedge \B A =_{\mathsmaller{\C E^{\Disj}(X)}} 0_{\B X}$ and
		$0_{\B A} \wedge \B A =_{\mathsmaller{\C E^{\Disj}(X)}} 0_{\B A}$.\\[1mm]
		\normalfont(vi)
		\itshape
		$1_{\B X} \vee \B A =_{\mathsmaller{\C E^{\Disj}(X)}} 1_{\B X}$ and $1_{\B A} \vee
		\B A =_{\mathsmaller{\C E^{\Disj}(X)}} \B 1_A$.\\[1mm]
		\normalfont(vii)
		\itshape  
		$1_{\B X} \wedge \B A =_{\mathsmaller{\C E^{\Disj}(X)}} \B A =_{\mathsmaller{\C E^{\Disj}(X)}} 1_{\B A} \wedge	\B A$.\\[1mm]
		\normalfont(viii)
		\itshape 
		$0_{-\B A} =_{\mathsmaller{\C E^{\Disj}(X)}} 0_{\B A}$.\\[1mm]
		\normalfont(ix)
		\itshape
		$0_{0_{\B X}} =_{\mathsmaller{\C E^{\Disj}(X)}} 0_{\B X}$ and
		$1_{1_{\B X}} =_{\mathsmaller{\C E^{\Disj}(X)}} 1_{\B X}$.\\[1mm]
		\normalfont(x)
		\itshape
		$0_{0_{\B A}} =_{\mathsmaller{\C E^{\Disj}(X)}} 0_{\B A} =_{\mathsmaller{\C E^{\Disj}(X)}} 0_{1_{\B A}}$.
		\end{proposition}

	\section{Boolean-valued, partial functions}
	\label{sec: bvpf}

	\begin{definition}\label{def: extpartialfunction}
		Let $X, Y$ be in $\SetExtIneq$. 
		A \textit{partial function}\index{partial function} from $X$ to $Y$ is a pair $\B f := (A, f)$, where $A =: \dom(\B f) \in \C E(X)$ and $f \colon A \to Y$ is a $($total$)$ function. We may also use the notation $\B f \colon X \pto Y$.
		We call $\B f$ total, if $\dom(\B f) =_{\C E(X)} X$. If $\B g := (B, g)$, where $B \in \C E(X)$ and $g \colon B \to Y$, let
		$$\B f \leq \B g :\TOT A \subseteq  B \ \wedge \ \forall_{x \in A}\big(f(x) =_Y g(y)\big).$$
		The \textit{partial function space} $\B F(X,Y)$ is
		equipped with the equality $\B f =_{\B F(X,Y)} \B g :\TOT \B f \leq \B g \ \& \ \B g \leq \B f$.
		If $\neq_X$ and $\neq_Y$ are inequalities on $X, Y$, respectively, let
		$\B F^{\se}(X, Y)$ be the totality of strongly extensional, partial functions from $X$ to $Y$. Accordingly, $(\B F^{\se}(X, \two)) \ \B F(X, \two)$ is the totality
		of $($strongly extensional$)$ \textit{Boolean-valued, partial functions} on $X$. The partial $($Boolean$)$ dual set of $X$ is the totality $\B X^* := \B F(X,\two)$. Let $\sim \B f := 1 - \B f$, and
		$\B f + \B g$, $\B f \vee \B g := \max\{\B f, \B g\}$, $\B f \cdot \B g := \B f \wedge \B g := \min\{\B f, \B g\}$ on
		$\dom(\B f) \cap \dom(\B g)$, for every $\B f, \B g \in X^*$.
		The strongly extensional, partial dual of $X$ is the totality 
		$$\B X^{\circledast} := \{\B f \in \B X^*  \mid f \ \mbox{is strongly extensional}\}.$$
		Let $\B X^{\circledast \circledast} := \big[\B X^{\circledast}\big]^{\circledast}$ be the second strongly extensional, $($Boolean$)$ partial dual of $X$. 
		\end{definition}

\begin{remark}\label{rem: seops}
\normalfont (i)
\itshape If $\B A \in \C E^{\Disj}(X)$, then $\B \chi_{\B A}  := (\dom(\B A), \chi_{\B A}) \in \B X^{\circledast}$.\\[1mm]
\normalfont (ii)
\itshape If $\B f, \B g \in \B X^{\circledast}$, then $\sim \B f, \B f + \B g, \B f \vee \B g, \B f \cdot \B g$ and $\B f \wedge \B g$
are in $\B X^{\circledast}$.
\end{remark}
	
\begin{proof}
(i) Suppose that $\chi_{\B A}(x) =_{\D 2} 1$ and $\chi_{\B A}(y) =_{\D 2} 0$, for some $x, y \in \dom(\B A)$. Hence, $x \in A^1$ and $y \in A^0$, and since $A^1 \Disj A^0$, we get $x \neq_X y$.\\
(ii) We show only that $\B f + \B g \in \B X^{\circledast}$ and for the rest cases we work similarly. Let $x, y \in X$ such that
$f(x) + g(x) \neq_{\D 2} f(y) + g(y)$. By Remark~\ref{rem: boolese} $f(x) \neq_{\D 2} f(y)$ or $g(x) \neq_{\D 2} g(y)$, hence by strong extensionality of $f$ and $g$ we get in both possible cases that $x \neq_{\D 2} y$.
\end{proof}

	If $\neq_Y$ is an inequality on $Y$, an inequality on
	$\B F(X,Y)$ was defined in~\cite{MWP24} by
	$$\B f =_{\mathsmaller{\B F(X,Y)}} \B g :\TOT X_P \neq_{\mathsmaller{\C E(X)}} X_Q \vee \big[X_P =_{\mathsmaller{\C E(X)}} X_Q \wedge \exists_{x \in X_P}\big(f(x) \neq_Y g(x)\big)\big].$$
	Next follows an equivalent formulation of the equality between Boolean-valued partial functions that allows the definition of a corresponding inequality relation between them in a natural way.

\begin{definition}\label{def: equalpf}
Let the equality and inequality on $\B X^*$ $($and similarly on $\B X^{\circledast})$ given by 
\begin{align*}
\B f =_{\B X^*} \B g & :\TOT \forall_{x \in X}\bigg[\big(x \in \dom(\B f) \To x \in \dom(\B g)\big) \wedge \\
& \ \ \ \ \ \ \ \ \ \ \ \ \ \big(x \in \dom(\B g) \To x \in \dom(\B f)\big) \wedge \\
& \ \ \ \ \ \ \ \ \ \ \ \ \ \big(x \in \dom(\B f) \cap \dom(\B g) \To f(x) =_{\two} g(x)\big)\bigg],
\end{align*}
\begin{align*}
	\B f \neq_{\B X^*} \B g & :\TOT \exists_{x \in X}\bigg[\big(x \in \dom(\B f) \wedge x \in \dom(\B g)^{\neq_X}\big) \vee \\
	& \ \ \ \ \ \ \ \ \ \ \ \ \ \big(x \in \dom(\B g) \wedge x  \in \dom(\B f)^{\neq_X}\big) \vee \\
	& \ \ \ \ \ \ \ \ \ \ \ \ \ \big(x \in \dom(\B f) \cap \dom(\B g) \wedge f(x) \neq_{\two} g(x)\big)\bigg].
\end{align*}
If $\B f \neq_{\B X^*} \B g$ and $x_0 \in X$, such that $x_0$ witnesses the inequality $\B f \neq_{\B X^*} \B g$, we write $x_0 \colon \B f \neq_{\B X^*} \B g$.
\end{definition}

A weak version of the the definition of $f \neq_{\B X^*} g$ would employ weak negation, as follows:
\begin{align*}
	& \exists_{x \in X}\bigg[\big(x \in \dom(\B f) \wedge x \notin \dom(\B g)\big) \vee \\
	& \ \ \ \ \ \ \ \  \big(x \in \dom(\B g) \wedge x \notin \dom(\B f)\big) \vee \\
	& \ \ \ \ \ \ \ \  \big(x \in \dom(\B f) \cap \dom(\B g) \wedge f(x) \neq_{\two} g(x)\big)\bigg].
\end{align*}
Notice the similarity between Definition~\ref{def: equalpf} and the definition of equality and inequality between functions in Definition~\ref{def: function} and Definition~\ref{def: canonicalineq}, respectively. 
In general, we cannot show that  $\B f \neq_{\B X^*} \B g$ is an apartness relation, although there are not many examples of this phenomenon\footnote{Another such inequality is the canonical inequality on the exterior union of a family of sets (see~\cite{Pe20}, p.~43).}. If $\B f \neq_{\B X^*} \B g$ was an apartness relation, it would be an extensional inequality (see~\cite{Pe20}, p.~11).  Moreover, we cannot prove, in general, that the inequality $\B f \neq_{\B X^*} \B g$ is tight. It is straightforward to show though, that if $\dom(\B f)$ and $\dom(\B g)$ are tight subsets\footnote{A subset $A$ of $X \in \SetExtIneq$ is called tight, if $(A^{\neq_X})^c \subseteq A$. Tight subsets are introduced in~\cite{Pe24b} and are related to tight formulas studied in~\cite{KP23}.}, then $\neg\big(\B f \neq_{\B X^*} \B g\big) \To \B f =_{\B X^*} \B g$, as $\neq_{\D 2}$ is a tight inequality. Next we prove the extensionality of $\neq_{\B X^*}$, and hence the extensionality of $\neq_{\B X^{\circledast}}$.

\begin{proposition}\label{prp: ext1}
\label{prp: ineqext} 
The inequality relation $f \neq_{\B X^*} g$ is extensional.
\end{proposition}

\begin{proof}
If $\B f =_{\B X^*} \B f{'}$, $\B g =_{\B X^*} \B g{'}$, and $\B f \neq_{\B X^*} \B g$, we show that $\B f{'} \neq_{\B X^*} \B g{'}$. If $x_0 \colon \B f \neq_{\B X^*} \B g$, we show that 
$x_0 \colon \B f{'} \neq_{\B X^*} \B g{'}$. Suppose first that $x_0 \in \dom(\B f) \wedge x_0 \in \dom(\B g)^{\neq_X}$. As $\B f =_{\B X^*} \B f{'}$, we get  $x_0 \in \dom(\B f{'})$.
As $x_0 \in \dom(\B g)^{\neq_X}$, by the equality $\B g =_{\B X^*} \B g{'}$ we have that $\dom(\B g) =_ {\C E(X)} \dom(\B g{'})$, and hence $x_0 \in \dom(\B g)^{\neq_X}$ too. If $x _0 \in \dom(\B g) \wedge x_0  \in \dom(\B f)^{\neq_X}$,
we work similarly. If $x_0 \in \dom(\B f) \cap \dom(\B g) \wedge f(x_0) \neq_{\two} g(x_0)$, then by the supposed equalities we have that $x_0 \in \dom(\B f{'}) \cap \dom(\B g{'})$. As $x_0 \in \dom(\B f) \cap \dom(\B f{'})$ and $x_0 \in \dom(\B g) \cap \dom(\B g{'})$, we get $f{'}(x_0) =_{\two} f(x_0) \neq_{\two} g(x_0) =_{\two} g{'}(x_0)$.
\end{proof}

 The above definitions of equality and inequality for Boolean-valued, partial functions induce a notion of equality and inequality between complemented subsets. 
 For the case of equality, we recover the equality of complemented subsets given in Definition~\ref{def: extapartsubsets}. For the case of inequality we get a property similar to the following definition of inequality between extensional subsets of a set.
 
 \begin{definition}\label{def: subineq}
 If $X \in \SetExtIneq$, let its extensional subsets $\C E(X)$ be equipped with the inequality
 $$A \neq_{\C E(X)} B :\TOT \exists_{x \in X}\bigg[\big(x \in A \wedge x \in B^{\neq_X}\big) \vee\big(x \in B \wedge x \in A^{\neq_X}\big)\bigg].$$
 \end{definition}
 
 It is immediate to check that the inequality $A \neq_{\C E(X)} B$ is extensional.

 \begin{definition}\label{def: csineq} Let 
 		$\B A \simeq_{\C E^{\Disj}(X)} \B B  :\TOT \B \chi_{\B A} =_{X^{\circledast}} \B \chi_{\B B}$, and 	let
 	$\B A \neq_{\C E^{\Disj}(X)} \B B :\TOT \chi_{\B A} \neq_{X^{\circledast}} \chi_{\B B}$.
 	 \end{definition}

 	\begin{remark}\label{rem: equivs}
 		\normalfont (i)
 		\itshape $\B A \simeq_{\C E^{\Disj}(X)} \B B  \TOT \B A =_{\C E^{\Disj}(X)} \B B$.\\[1mm]
 			\normalfont (ii)
 		\itshape $\B A \neq_{\C E^{\Disj}(X)} \B B$ if and only if the following condition holds:
 	\begin{align*}
 		& \exists_{x \in X}\bigg[\big(x \in A^1 \cup A^0 \wedge x \in (B^1 \cup B^0)^{\neq_X}\big) \vee \\
 		& \ \ \ \ \ \ \ \ \big(x \in B^1 \cup B^0 \wedge x \in (A^1 \cup A^0)^{\neq_X}\big) \vee \\
 		& \ \ \ \ \ \ \ \ \big(x \in (A^1 \cup A^0) \cap (B^1 \cup B^0) \wedge \chi_{\B A}(x) \neq_{\two} \chi_{\B B}(x)\big)\bigg].
 	\end{align*}
 	\end{remark}

 	\begin{proof}
 	(i)  By Definition~\ref{def: equalpf} we have that $\B \chi_{\B A} =_{X^{\circledast}} \B \chi_{\B B}$ if and only if
 	\begin{align*}
 		& \forall_{x \in X}\bigg[\big(x \in A^1 \cup A^0 \To x \in B^1 \cup B^0\big) \wedge \\
 		& \ \ \ \ \ \ \ \ \ \ \ \ \ \big(x \in B^1 \cup B^0 \To x \in A^1 \cup A^0\big) \wedge \\
 		& \ \ \ \ \ \ \ \ \ \ \ \ \ \big(x \in (A^1 \cup A^0) \cap (B^1 \cup B^0) \To \chi_{\B A}(x) =_{\two} \chi_{\B B}(x)\big)\bigg],
 	 	\end{align*}
 	 	which is trivially equivalent to $\B A =_{\C E^{\Disj}(X)} \B B$.\\
 	(ii) It follows immediately by Definition~\ref{def: equalpf}.
 	\end{proof}
 	
 \begin{definition}\label{def: notation} If $\B A \neq_{\C E^{\Disj}(X)} \B B$, and if, according to Remark~\ref{rem: equivs}, $x_0 \in X$ witnesses this inequality, then we may write $x_0 \colon \B A \neq_{\C E^{\Disj}(X)} \B B$.
\end{definition}

\begin{proposition}\label{prp: ext2}
	Let $\B A, \B B \in \C E^{\Disj}(X)$.\\[1mm]
	\normalfont (i)
	\itshape The inequality relation $\B A \neq_{\C E^{\Disj}(X)} \B B$ is extensional.\\[1mm]
		\normalfont (ii)
	\itshape $\B A$ is total if and only if $1_{\B A}  =_{\C E^{\Disj}(X)} 1_{\B X}$.\\[1mm]
	\normalfont (iii)
	\itshape $\B A  \neq_{\C E^{\Disj}(X)} 0_{\B X} \TOT \exists_{x \in X}\big(x \in \dom(\B A)^{\neq_{X}} \vee x \in A^1\big)$.\\[1mm]
	\normalfont (iv)
	\itshape $\B A  \neq_{\C E^{\Disj}(X)} 1_{\B X} \TOT \exists_{x \in X}\big(x \in \dom(\B A)^{\neq_{X}} \vee x \in A^0\big)$.\\[1mm]
		\normalfont (v)
	\itshape $A$ is inhabited if and only if $A \neq_{\C E^{\Disj}(X)} 0_{\B A}$.\\[1mm]
		\normalfont (vi)
	\itshape $A$ is coinhabited if and only if $A \neq_{\C E^{\Disj}(X)} 1_{\B A}$.
\end{proposition}

\begin{proof}
(i) By Proposition~\ref{prp: ext1}, Definition~\ref{def: csineq} and the equivalence 
	$ \B A =_{\C E^{\Disj}(X)} \B B \TOT \B \chi_{\B A} =_{X^{\circledast}} \B \chi_{\B B}$.\\
	Case (ii) is immediate to show, and  cases (iii) - (v) follow easily by Remark~\ref{rem: equivs}.
\end{proof}

\begin{remark}\label{rem: basicequal} If $\B A, \B B \in \C E^{\Disj}(X)$, and if $\square \in \{\vee, \wedge\}$, then $\dom(\B A) \cap \dom(\B B) =_{\C E(X)} \dom(\B A \square \B B)$.
\end{remark}
	
\begin{proof}
By straightforward calculations.
\end{proof}

\section{The partial Gelfand transform}
\label{sec: gelfand}
	
As the powerset of $X \in \Set$ is classically bijective to the total dual set of $X$ i.e., the total Boolean-valued functions on $X$, the complemented powerset of $X \in \SetExtIneq$ is bijective to its strongly extensional partial dual $\B X^{\circledast}$
(for the proof of the following fact see~\cite{MWP24}, Proposition 2.14).

\begin{proposition}\label{prp: extcompl2}
If $\C X := (X, =_X, \neq_X) \in \SetIneq$, let the assignment routines
	$$\chi \colon \C E^{\Disj}(X) \sto \B X^{\circledast} \ \ \& \ \
	\delta \colon \B X^{\circledast} \sto \C E^{\Disj}(X),$$
	$$\B A \mapsto \chi(\B A) \ \ \ \B A \in \C E^{\Eisj}(X),$$
	$$\B f_A \mapsto \delta(\B f_A) := \big(\delta^1( \B f_A), \delta^0( \B f_A)\big), \ \ \ \B f_{A} \in \B X^{\circledast},$$
$$ \delta^1( \B f_A) := \big\{a \in A \mid  f_A(a) =_{\mathsmaller{\two}} 1 \big\}, \ \ \ \ \ 
	\delta^0( \B f_A) := \big\{a \in A \mid  f_A(a) =_{\mathsmaller{\two}} 0 \big\}.$$
Then $\chi, \delta$ are well-defined functions that are inverse to each other. 
\end{proposition}

\noindent
Proposition~\ref{prp: extcompl2} indicates that there is a way to constructivise many classical results concerning
\textbf{points, subsets, and total functions} by employing 
\textbf{points, complemented subsets, and partial functions}.
As another example of this method, we describe the constructive and partial counterpart to the classical embedding of a set $X$ to its total, second dual. 
If $X \in \SetExtIneq$, the \textit{total Gelfand transform}\footnote{In this section the Gelfand transform, both its total and partial version, concerns the embedding of $X$ into the second Boolean dual of $X$. A similar treatment of the Gelfand transform that concerns the embedding of $X$ into the second real dual of $X$ is possible, where the real dual of $X$, total or partial, involves  total or partial functions of $X$ to the reals.} on $X$ is the assignment routine 
$${}^{\hat{}} \colon X \sto \D F(\D F(X, \two), \two), \ \ \ x \mapsto \hat{x}, \ \ \ x \in X,$$ 
$$ \hat{x} \colon \D F(X, \two) \to \two, \ \ \ 
\hat{x}(f) := f(x), \ \ \ f \in \D F(X, \two).$$
It is immediate to see (also constructively) that the total Gelfand transform on $X$ is a well-defined function, but the proof that $\D F(X, \two)$ \textit{separates the points} of $X$ i.e.,
$$\forall_{f \in \D F(X, \two)}\big(f(x) =_{\two} f(y)\big) \To x =_X y,$$
with which one proves that the total Gelfand transform is an embedding, rests on $\PEM$. It is with $\PEM$ that one proves the totality of the required function $\bar{x} \colon X \to \two$, where 
$$\bar{x}(x{'}) := \left\{ \begin{array}{ll}
	1   &\mbox{, $x{'} =_X x$}\\
	0             &\mbox{, $\neg(x{'} =_X x)$.}
\end{array}
\right. $$
To avoid the use of $\PEM$, one needs to use partial Boolean-valued functions. Next follows the constructive version of the classical embedding property of the total Gelfand transform. For that, one has to replace the total Gelfand transform by its partial version. The next proof is the constructive version of the classical proof of the fact that the total Gelfand transform is an embedding.

\begin{proposition}\label{prp: pGelfand}
If $\C X \in \SetExtIneq$, the partial Gelfand transform on $X$ is the assignment routine 
$${}^{\widehat{}} \colon X \sto \B X^{\circledast \circledast}, \ \ \ x \mapsto \widehat{\B x}, \ \ \ x \in X,$$ 
$$  \widehat{\B x} := \big(\B X^{\circledast}_x, \widehat{x}\big), \ \ \ \widehat{x} \colon \B X^{\circledast}_x \to \two, \ \ \  
\B X^{\circledast}_x := \{\B f \in \B X^{\circledast} \mid x \in \dom(\B f)\},$$
$$ \widehat{x}(\B f) := f(x), \ \ \ \B f \in \B X^{\circledast}_x.$$
The partial Gelfand transform on $X$ is a well-defined embedding.
\end{proposition}

\begin{proof}
To show that the partial Gelfand transform is well-defined, it suffices to show that $\widehat{\B x}$ is a 
strongly extensional partial function.
Let $\B f, \B g \in \B X^{\circledast}_x$ i.e., $x \in \dom(\B f) \cap \dom(\B g)$, such that $\widehat{\B x}(f) \neq_{\two} \widehat{\B x}(\B g) :\TOT f(x)  \neq_{\two} g(x)$.
Then by Definition~\ref{def: equalpf} we conclude that $x \colon \B f \neq_{\B X^{\circledast}} \B g$, and hence $x \colon \B f \neq_{\B X^{\circledast}_x} \B g$.
Next we show that the partial Gelfand transform is a function i.e., if $x =_X y$, then 
\begin{align*}
	\widehat{\B x} =_{\B X^{\circledast \circledast}} \widehat{\B y} & :\TOT \forall_{\B f \in \B X^{\circledast}}\bigg[\big(\B f \in \dom(\widehat{\B x}) \To \B f \in \dom(\widehat{\B y})\big) \wedge \\
	& \ \ \ \ \ \ \ \ \ \ \ \ \ \ \ \ \big(\B f \in \dom(\widehat{\B y}) \To \B f \in \dom(\widehat{\B x})\big) \wedge \\
	& \ \ \ \ \ \ \ \ \ \ \ \ \ \ \ \ \big(\B f \in \dom(\widehat{\B x}) \cap \dom(\widehat{\B y}) \To \widehat{x}(\B f) =_{\two} \widehat{y}(\B f)\big)\bigg].
\end{align*}
Let $\B f \in \B X^{\circledast}$. If $\B f \in \dom(\widehat{\B x}) :\TOT x \in \dom(\B f)$, then by the extensionality of the subset $\dom(\B f)$ of $X$ we get $y \in \dom(\B f) \TOT: \B f \in \dom(\widehat{\B y})$. The second implication is shown similarly. If $\B f \in \dom(\widehat{\B x}) \cap \dom(\widehat{\B y}) \TOT x \in \dom(\B f) \wedge y \in \dom(\B f)$, then $x =_X y \To f(x) =_{\two} f(y)$, which is equivalent to $\widehat{x}(\B f) =_{\two} \widehat{y}(\B f)$. Finally, we show that the partial Gelfand transform is an embedding i.e., if $\widehat{\B x} =_{\B X^{\circledast \circledast}} \widehat{\B y}$, then $x =_X y$. Let the partial function $\widebar{\B x} : = \big(\{x\} \cup \{x\}^{\neq_X}, \widebar{x}\big)$, defined by 
$$\widebar{x}(z) := \left\{ \begin{array}{ll}
	1   &\mbox{, $z\in \{x\}$}\\
	0             &\mbox{, $z \in\{x\}^{\neq_X}$.}
\end{array}
\right. $$
Clearly, $x \in \dom(\widebar{\B x})$. To show that $\widebar{x}$ is strongly extensional, let $z, z{'} \in \dom(\widebar{\B x})$ such that $\widebar{\B x}(z) =_{\two} 1$ and $\widebar{x}(z{'}) =_{\two} 0$ i.e., $z =_X x$ and $z{'} \neq_X x$. By the extensionality of $\neq_X$ we get $z{'} \neq_X z$. Hence,
$\widebar{\B x} \in \dom(\widehat{\B x})$. By definition of the equality $\widehat{\B x} =_{\B X^{\circledast \circledast}} \widehat{\B y}$ we have that $\widebar{\B x} \in \dom(\widehat{\B y})$, and hence $\widebar{\B x} \in \dom(\widehat{\B x}) \cap \dom(\widehat{\B y})$. Consequently, 
$\widehat{x}(\widebar{\B x}) =_{\two} \widehat{y}(\widebar{\B x})$, thus $1 =: \widebar{x}(x) =_{\two} \widebar{x}(y)$, from which we get $y =_X x$.
\end{proof}

Although the total Gelfand transform is strongly extensional, if we restrict on strongly extensional functions, we cannot show this for the partial Gelfand transform.  According to Definition~\ref{def: equalpf}, 
\begin{align*}
	\widehat{\B x} \neq_{\B X^{\circledast \circledast}} \widehat{\B y} & :\TOT \exists_{\B f \in \B X^{\circledast}}\bigg[\big(\B f \in \dom(\widehat{\B x}) \wedge f \in \dom(\widehat{\B y})^{\neq_{\B X^{\circledast}}}\big) \vee \\
	& \ \ \ \ \ \ \ \ \ \ \ \ \ \ \ \ \big(\B f \in \dom(\widehat{\B y}) \wedge \B f \in \dom(\widehat{\B x})^{\neq_{\B X^{\circledast}}}\big) \vee \\
	& \ \ \ \ \ \ \ \ \ \ \ \ \ \ \ \  \big(\B f \in \dom(\widehat{\B x}) \cap \dom(\widehat{\B y}) \wedge \widehat{x}(\B f) \neq_{\two} \widehat{y}(\B f)\big)\bigg].
\end{align*}
If $\B f_0 \colon \widehat{\B x} \neq_{\B X^{\circledast \circledast}} \widehat{\B y}$, and as $f_0$ is strongly extensional, it is only within the last disjunct that we get $x \neq_X y$. We need a stronger formulation of $\B f \notin \dom(\widehat{\B x})$ or $\B f \notin \dom(\widehat{\B y})$, in order to get the strong extensionality of the partial Gelfand transform.

\section{Swap algebras}
\label{sec: swap}

Proposition~\ref{prp: isswapa1}(iii) shows that the algebra of complemented subsets has, in general, many unequal zeros and ones, and hence it cannot be, in general, a Boolean algebra. Actually, the totality of complemented subsets of some inhabited $\C X := (X, =_X, \neq_X)$ in $\SetExtIneq$ is a \textit{swap algebra} of type $(\ti)$, if the first kind of definition of join and meet is considered, and it is a swap algebra of  type $(\tii)$, if the second kind of definition of join and meet is considered (see~\cite{MWP24} for all necessary details). The abstract version of the common properties of the two algebras of complemented subsets $\C B_1^e(X)$ and $\C B_2^e(X)$ forms the notion of a swap algebra, where the axiom $(\swapa_{\ti})$ in the following Definition~\ref{def: swapalgebra} is satisfied by $\C B_1^e(X)$ and the axiom $(\swapa_{\tii})$ is satisfied by $\C B_2^e(X)$. A swap algebra of type $(\ti)$ does not satisfy, in general, condition $(\swapa_{\tii})$, and a swap algebra of type $(\tii)$ does not satisfy, in general, condition $(\swapa_{\tii})$.

\begin{definition}\label{def: swapalgebra}
	A \textit{swap algebra} is a structure $\C A := \big(A, \vee, \wedge, 0, 1, -, 0_{-}, 1_{-}\big)$,
	where $(A, =_A, \neq_A) \in \SetExtIneq$, $0, 1 \in A$, $\vee, \wedge, \colon A \times A \to A$ and $0_{-}, 1_{-} \colon A \to A$ 
	are functions, such that the following conditions hold:\\[1mm]
	$(\swapa_1)$ \ $a \vee a =_{\mathsmaller{A}} a$.\\[1mm]
	$(\swapa_2)$ \ $a \vee b =_{\mathsmaller{A}} b \vee a$.\\[1mm]
	$(\swapa_3)$ \ $a \vee (b  \vee c) =_{\mathsmaller{A}} (a \vee b) \vee c$.\\[1mm]
	$(\swapa_4)$ \ $a \vee (b  \wedge c) =_{\mathsmaller{A}} (a \vee b) \wedge (a \vee c)$.\\[1mm]
	$(\swapa_5)$ $0 \neq_{\mathsmaller{A}} 1$.\\[1mm]
	$(\swapa_6)$ \ $-0 =_{\mathsmaller{A}} 1$.\\[1mm]
	$(\swapa_7)$ \ $-(-a) =_{\mathsmaller{A}} a$.\\[1mm]
	$(\swapa_8)$ \ $-(a \vee b) =_{\mathsmaller{A}} (-a) \wedge (-b)$.\\[1mm]
	$(\swapa_9)$ \ $a \vee (-a) =_{\mathsmaller{A}} 1_a$ and $a \wedge (-a) =_{\mathsmaller{A}} 0_a$.\\[1mm]
	$(\swapa_{10})$  \ $a  =_{\mathsmaller{A}} 0_{a} \vee a$.\\[1mm]
	A swap algebra $\C A$ is of \textit{type} $(\ti)$, if the following condition holds:\\[1mm]
	$(\swapa_{\ti})$ \ $(a \vee b) \wedge a =_{\mathsmaller{A}} a$.\\[1mm]
	A swap algebra  $\C A$ is of \textit{type} $(\tii)$, if the following condition holds:\\[1mm]
	$(\swapa_{\tii})$ \ $0_a \vee b =_{\mathsmaller{A}} 1_a \wedge b$.\\[1mm]
	We call an element $a $ of $A$:\\[1mm]
	\normalfont (i)
	\itshape total, if $1_a =_{\mathsmaller{A}} 1$.\\[1mm]
		\normalfont (ii)
	\itshape a unit, if $a =_{\mathsmaller{A}} 1_a$.\\[1mm]
	\normalfont (iii)
	\itshape inhabited, if $a \neq_{\mathsmaller{A}} 0_a$.\\[1mm]
		\normalfont (iv)
	\itshape coinhabited, or a counit, if $a \neq_{\mathsmaller{A}} 1_a$.
\end{definition}

The hypothesis of inhabitedness of $X$ is crucial to the proof that its complemented subsets satisfy condition $(\swapa_5)$. To show the inequality $0_{\B X} \neq_{\mathsmaller{\C E^{\Disj}(X)}} 1_{\B X}$, we use Definition~\ref{def: csineq}, where if $x_0 \in X$, then $x_0 \colon (X, \emptys_X) \neq_{\C E^{\Disj}(X)} (\emptys_X, X)$. For alternative and more dense characterisations of swap algebras see~\cite{MWP24}, sections 7 and 9, while for various examples of swap algebras of both types see~\cite{MWP24}, section 7. Swap algebras can look quite differently from sets of pairs of appropriately disjoint objects. In~\cite{Ha74}, p.~7, a Boolean structure is associated to the set $D_m$ of all positive integral divisors of $m$, where $m $ is a square-free natural number i.e., $m > 1$ and $m$ has a prime decomposition of the form $m =_{\Nat} p_1^{l_1} \cdot \ldots \cdot p_k^{l_k}$ and $l_1 =_{\Nat} \ldots =_{\Nat} l_k =_{\Nat} 1$. It turns out that in the general case, where $m$ is an arbitrary natural number $> 1$, not necessarily square-free, $D_m$ has the structure of a swap algebra of type $(\ti)$. For the proof of the following fact we refer to~\cite{MWP24}, Proposition 7.6.

\begin{proposition}\label{prp: nat}
	If $m > 1$, and $D_m := \{n \in \Nat \mid n | m\}$, let $\B 0 := 1, \B 1 := m$, and let $n \vee n{'} := \lcm \{n, n{'}\}$ and 
	$n \wedge n{'} := \gcdi \{n, n{'}\}$, the least common multiple and the greatest common divisor of $n, n{'}$, respectively. Let also $-n := \frac{m}{n}$, $0_n := n \wedge (-n)$, and
	$1_n := n \vee (-n)$, for every $n, n{'} \in D_m$. Then $\C D_m := (D_m, \vee, \wedge, \B 0, \B 1, -, 0_{-}, 1_{-})$ is a swap algebra of type $(\ti)$.
\end{proposition}

Notice that the swap algebra-structure of $\C D_m$ is interesting both from a constructive and a classical point of view! The next three propositions give a flavour of the common and non-common behaviour between swap algebras of type $(\ti)$ and $(\tii)$, and are straightforward to show.

\begin{proposition}\label{prp: swapalg1}
	If $\C A$ is a swap algebra, the following hold:\\[1mm]	
	$(1)$ \ $a \wedge a =_{\mathsmaller{A}} a$.\\[1mm]
	$(2)$ \ $a \wedge b =_{\mathsmaller{A}} b \wedge a$.\\[1mm]
	$(3)$ \ $a \wedge (b \wedge c) =_{\mathsmaller{A}} (a \wedge b) \wedge c$.\\[1mm]
	$(4)$ \ $a \wedge (b  \vee c) =_{\mathsmaller{A}} (a \wedge b) \vee (a \wedge c)$.\\[1mm]
	$(5)$ $-0_a =_{\mathsmaller{A}} 1_a$.\\[1mm]
	$(6)$ \ $0_{-a} =_{\mathsmaller{A}} 0_{a}$.\\[1mm]
	$(7)$ \ $0_0 =_{\mathsmaller{A}} 0$ and $1_1 =_{\mathsmaller{A}} 1$.\\[1mm]
	$(8)$  \ $0_{0_a} =_{\mathsmaller{A}} 0_a =_{\mathsmaller{A}} 0_{1_a}$. \\[1mm]
	$(9)$  \ $0 \vee a =_{\mathsmaller{A}} a$. \\[1mm]
	$(10)$  \ $0_{a}  \wedge a =_{\mathsmaller{A}} 0_a$. \\[1mm]
	$(11)$ \ $1 \wedge a  =_{\mathsmaller{A}} a  =_{\mathsmaller{A}} 1_a \wedge a$.\\[1mm]
	$(12)$ \ $1_a \vee a =_{\mathsmaller{A}} 1_a$.
\end{proposition}

\begin{proposition}\label{prp: swapalgI1}
	If $\C A$ is a swap algebra of type $(\ti)$, the following hold:\\[1mm]	
	$(1)$ \ $0 \wedge a =_{\mathsmaller{A}} 0$.\\[1mm]
	$(2)$ \ $1 \vee a  =_{\mathsmaller{A}} 1$.\\[1mm]
	$(3)$ \ $ a =_{\mathsmaller{A}} (a \wedge b) \vee a$.\\[1mm]
	$(4)$ The restriction of all operations of $\C A$ to its subset $\{0, 1\}$ is isomorphic to the Boolean algebra of $\two$.
\end{proposition}

\begin{proposition}\label{prp: swapalgII1}
	If $\C A$ is a swap algebra of type $(\tii)$, the following hold:\\[1mm]	
	$(1)$ \ $a \wedge (-a \vee b) =_{\mathsmaller{A}} a \wedge b$.\\[1mm]
	$(2)$ \ $a \vee (-a \wedge b) =_{\mathsmaller{A}} a \vee b$.\\[1mm]
	$(3)$ \ $0 \wedge a =_{\mathsmaller{A}} 0_a$.\\[1mm]
	$(4)$ \ $1 \vee a  =_{\mathsmaller{A}} 1_a $.\\[1mm]
	$(5)$ \ $0_{a \vee b} =_{\mathsmaller{A}} 0_{a \wedge b} =_{\mathsmaller{A}} 0_a \vee 0_b =_{\mathsmaller{A}} 0_a \wedge 0_b$.\\[1mm]
	$(6)$ \ $0_a \wedge b =_{\mathsmaller{A}} 0_a \wedge (-b) =_{\mathsmaller{A}} 0_{a \wedge b}$.\\[1mm]
	$(7)$ \ $(a \vee b) \wedge a =_{\mathsmaller{A}} 1_b \wedge a =_{\mathsmaller{A}} (a \wedge b) \vee a$.
	
\end{proposition}

In analogy to the duality principle for Boolean algebras, we have the following \textit{duality principle for swap algebras} $(\DPSA)$: if an equation $(e)$ holds, then the equation $(e^*)$ also holds, where 
$(e^*)$ follows from $(e)$ by interchanging $0$ and $1$, $0_r$ with $1_r$, $\wedge$ and $\vee$.
For the proof of the following fact we refer to~\cite{MWP24}, Proposition 7.10.

\begin{proposition}\label{prp: totala}
	If $\C A$ is a swap algebra of type $(\ti)$ or $(\tii)$, then its total elements $T(A)$ form a swap algebra both of type $(\ti)$ and $(\tii)$. Actually, 
	$T(A)$ is a  Boolean algebra. Conversely, every Boolean algebra is a swap algebra both of type $(\ti)$ and $(\tii)$.
	
\end{proposition}

Swap algebras of type $(\tii)$ suit better  related to Boolean-valued partial functions rather than swap algebras of type $(\ti)$ (see~\cite{MWP24}). The abstract version of the Boolean-valued partial functions on a set with an extensional inequality is the notion of the so-called \textit{swap ring} in~\cite{MWP24} or \textit{Boolean rig} in~\cite{MWP24b}. In~\cite{MWP24}, Proposition 8.6, it is shown that the duality between swap algebras of type $(\tii)$ and swap rings is a generalisation of the duality between Boolean algebras and Boolean rings. It is not an accident that it is for the (separated) swap algebras of type $(\tii)$ that we have a representation theorem of Stone-type (see section~\ref{sec: repr}). In~\cite{MWP24b} it is shown that given a commutative rig i.e., a structure $(S, +, \cdot, 0, 1)$ such that both reducts $(S, +, 0)$ and $(S, \cdot, 1)$ are commutative monoids and $\cdot$ is both left- and right-distributive over $+$, then its orthogonal idempotents $O(S)$ from a swap ring in a canonical way. Moreover, every Boolean ring $R$ is isomorphic to the total elements of $O(R)$. This is another example of interesting new mathematics with constructive origin!

\section{Swap characters}
\label{sec: shomos}

The theory of partial homomorphisms between various algebraic structures is less developed than the theory of total homomorphisms between them\footnote{The papers~\cite{Pe66, AS70, SHT03} is a short, non-exhaustive list of papers related to the study of partiality in algebra.}. In this section we introduce a notion of a partial, Boolean-valued swap-homomorphism that fits to the proof of our representation theorem in section~\ref{sec: repr}. \textit{Throughout this section $A, B, C$ are swap algebras of type $(\tii)$}\footnote{A field of elements $F$ in a swap algebra of type $(\ti)$ is defined by $(\fii_1, \fii_2)$ and $(\fii_3{'})$, where according to the latter, if $a, b \in F$, then $a \vee b \in F$. A swap character on $A$ is defined as in Definition~\ref{def: partialshomo}.}.

\begin{definition}\label{def: partialshomo}
A field of elements in $A$ is a subset $F$ of $A$, satisfying the following conditions:\\[1mm]
$(\fii_1)$ $1 \in F$.\\[1mm]
$(\fii_2)$ If $a \in F$, then $-a \in F$.\\[1mm]
$(\fii_3)$ $a \vee b \in F$ if and only if $a \in F$ and $b \in F$.\\[1mm]
We call an element $\B f := \big(\dom(\B f), f\big)$ of $\B A^{\circledast}$ a swap character on $A$, if $\dom(\B f)$ is a field of elements in $A$, and for every $a, b \in \dom(\B f)$ the following conditions hold:\\[1mm]
$(\ho_1)$ $f(1) =_{\two} 1$.\\[1mm]
$(\ho_2)$ If $a \in \dom(\B f)$, then $f(-a) =_{\two} \ \sim {\hspace{-1mm}}f(a)$.\\[1mm]
$(\ho_3)$ If $a, b \in \dom(\B f)$, then $f(a \vee b) =_{\two} f(a) \vee f(b) $.\\[1mm]
Let $\widehat{\B A}$ be the set of swap characters on $A$ equipped with the equality and inequality inherited from $\B A^{\circledast}$. If $B$ is a swap algebra of type $(\tii)$, let $\widehat{(\B A, \B B)}$ be the set of partial swap homomorphisms from $A$ to $B$, which are defined similar.
\end{definition}

If $A$ was a swap algebra of type $(\ti)$, then $1 =_A 1 \vee a$, hence $(\fii_3)$ would imply that a field of elements in $A$ is exactly $A$. On the other hand, if $A$ is a swap algebra of type $(\tii)$, then $1 \vee a =_A 1_a$, which is not, in general, equal in $A$ to $1$. If $F$ is a field of elements in $A$, then $0 \in F$, and if $a \in F$, then $1_a, 0_a \in F$ too. Clearly, $A$ itself is a field of elements in $A$, and if $F, G$ are fields of elements in $A$, then $F \cap G$ is also a field of elements in $A$.

\begin{lemma}\label{lem: ex1}
	If $\C X \in \SetExtIneq$ and $x_0 \in X$, then 
	$$\C E^{\Disj}(X, x_0) := \big\{\B A\in \C E^{\Disj}(X) \mid x_0 \in \dom(\B A)\big\}$$
	is a field of elements in the swap algebra of type $(\tii)$ $\C E^{\Disj}(X)$, and the pair
	$$\widehat{\B x_0} := \big(\C E^{\Disj}(X, x_0), \widehat{x_0}\big), \ \ \ \widehat{x_0} \colon \C E^{\Disj}(X, x_0) \to \two, \ \ \  \B A \mapsto \widehat{x_0}(\B A), \ \ \ \B A\in \C E^{\Disj}(X, x_0),$$
	$$\widehat{x_0}(\B A) := \left\{ \begin{array}{ll}
		1   &\mbox{, $x_0 \in A^1$}\\
		0             &\mbox{, $x_0 \in A^0$,}
	\end{array}
	\right. $$
	is a swap character on $\C E^{\Disj}(X)$.
	\end{lemma}

\begin{proof}
Clearly, $(X, \emptys_X) \in \C E^{\Disj}(X, x_0)$, and if $\B A \in \C E^{\Disj}(X, x_0)$, then $-\B A \in \C E^{\Disj}(X, x_0)$. Next, we show that if
$\B A, \B B \in \C E^{\Disj}(X, x_0)$, then $\B A\vee \B B \in \C E^{\Disj}(X, x_0)$. By hypothesis $x_0 \in A^1 \cup A^0$ and $x_0 \in B^1 \cup B^0$, hence by Remark~\ref{rem: basicequal} we have that\footnote{This equality is satisfied only by the second algebra of complemented subsets.}
$x_0 \in (A^1 \cup A^0) \cap (B^1 \cup B^0) =_{\C E(X)} \dom(\B A \vee \B B)$. Conversely, if $x_0 \in \dom(\B A \vee \B B)$, then by Remark~\ref{rem: basicequal} again we have that $x_0 \in \dom(\B A)$ and $x_0 \in \dom(\B B)$ i.e., $\B A\in \C E^{\Disj}(X, x_0)$ and $\B B \in \C E^{\Disj}(X, x_0)$. Clearly, $\widehat{x_0}(\B A)(X, \emptys_X) := 1$, and $\widehat{x_0}(-\B A) := \sim\big(\widehat{x_0}(\B A)\big)$. By definition
$$\widehat{x_0}(\B A \vee \B B) := \left\{ \begin{array}{ll}
	1   &\mbox{, $x_0 \in (A^1 \cap B^1) \cup (A^1 \cap B^0) \cup (A^0 \cap B^1)$}\\
	0             &\mbox{, $x_0 \in A^0 \cap B^0$.}
\end{array}
\right. $$	
Moreover, $\widehat{x_0}(\B A) \vee \widehat{x_0}(\B B) := 1$, if $\widehat{x_0}(\B A) := 1 :\TOT x_0 \in A^1$ or  $\widehat{x_0}(\B B) := 1 :\TOT x_0 \in B^1$. As $x_0 \in \dom(\B A \vee \B B)$, this happens exactly when $x_0 \in (A^1 \cap B^1) \cup (A^1 \cap B^0) \cup (A^0 \cap B^1)$. Clearly, 
$\widehat{x_0}(\B A) \vee \widehat{x_0}(\B B) := 0$, if and only if $\widehat{x_0}(\B A) := 0 =: \widehat{x_0}(\B B)$ i.e., if and only if 
$x_0 \in A^0 \cap B^0$. Finally, we show that $\widehat{x_0}$ is strongly extensional. For that, let $\B A, \B B \in \C E^{\Disj}(X, x_0)$, such that  $\widehat{x_0}(\B A) \neq_{\two} \widehat{x_0}(\B B)$. As $x_0 \in (A^1 \cup A^0) \cap (B^1 \cup B^0)$, by Definition~\ref{def: csineq} we get that 
$x_0 \colon \B A \neq_{\C E^{\Disj}(X)} \B B$.
\end{proof}

Next follows the partial counterpart to the separation of the points of a set $X$ by a given set $F$ of Boolean-valued (total) functions on $X$, i.e., to the property
$$\forall_{x, y \in X}\bigg[\forall_{f \in F}\big(f(a) =_{\two} f(b)\big) \To x =_X y\bigg].$$

\begin{definition}\label{def: sep}
If $\Phi \subseteq A^{\circledast}$, we say that $A$ is $\Phi$-separated, or $\Phi$ separates the points of $A$, if 
\begin{align*}
	 & \forall_{a, b \in A}\bigg[\forall_{\B f \in \Phi}\bigg(\big(a \in \dom(\B f) \To b \in \dom(\B f)\big) \wedge  \\
	& \ \ \ \ \ \ \ \ \ \ \ \ \ \ \ \ \big(b \in \dom(\B f) \To a \in \dom(\B f)\big) \wedge \\
	& \ \ \ \ \ \ \ \ \ \ \ \ \ \ \ \ \big(a \in \dom(\B f) \wedge b \in \dom(\B f) \To f(a) =_{\two} f(b)\big)\bigg)\\
	& \ \ \ \ \ \ \ \ \ \ \ \ \ \ \ \ \To a =_A b\bigg].
\end{align*}
If $\widehat{A}$ separates the points of $A$, we call $A$ a separated swap algebra\footnote{A separated swap algebra of type $(\ti)$ is defined similarly. In the case of the Boolean algebra $\C D_m$, where $m > 1$ is square-free, a field in $D_m$ satisfies $(\fii_3)$, hence $(\fii_3{'})$ too, and a character on $D_m$ is a total Boolean-homomorphism. If $m = p_1 \cdot  \ldots \cdot p_k$, then using the characters $\widehat{p_i} \colon D_m \to \two$, where $i \in \{1, \ldots, k\}$,  and
$$\widehat{p_i}(n) := \left\{ \begin{array}{ll}
	1   &\mbox{, $p_i | n$}\\
	0             &\mbox{, $p_i | \frac{m}{n}$,}
\end{array}
\right. $$
one can show that the Boolean algebra $\C D_m$ is separated. If $p$ is a prime number, then, according to our definition in footnote 14, $D_{p^2} = \{1, p, p^2\}$ has two fields ($\{1, p^2\}$ and $D_{p^2}$) and one partial character defined on $\{1, p\}$. The swap algebra $D_{p^3} = \{1, p, p^2, p^3\}$ has two fields ($\{1, p^3\}$ and $D_{p^3}$) and two characters. The swap algebra $D_{p^4} = \{1, p, p^2, p^3, p^4\}$ has four fields ($\{1, p^4\}$, $\{1, p^2, p^4\}$, $\{1, p, p^3, p^4\}$, and $D_{p^4}$) and two characters. The combinatorics related to the proof that for an arbitrary $m > 1$ the swap algebra $\C D_m$ of type $(\ti)$ is separated are not going to be included here.}
\end{definition}

\begin{proposition}\label{prp: sep1}
The swap algebra $\C E^{\Disj}(X)$ of type $(\tii)$ is separated.
\end{proposition}

\begin{proof}
If $\B A, \B B \in \C E^{\Disj}(X)$, we suppose that
\begin{align*}
	& \forall_{\B f \in \widehat{\C E^{\Disj}(X)}}\bigg(\big(\B A \in \dom(\B f) \To \B B \in \dom(\B f)\big) \wedge  \\
	& \ \ \ \ \ \ \ \ \ \ \ \ \ \ \ \ \big(\B B \in \dom(\B f) \To \B A \in \dom(\B f)\big) \wedge \\
	& \ \ \ \ \ \ \ \ \ \ \ \ \ \ \ \ \big(\B A \in \dom(\B f) \wedge \B B \in \dom(\B f) \To f(\B A) =_{\two} f(\B B)\big)\bigg).
\end{align*}
If $x_0 \in A^1$, then by Lemma~\ref{lem: ex1} we have that $\widehat{\B x_0} \in  \widehat{\C E^{\Disj}(X)}$, and the hypothesis $x_0 \in A^1$ implies that $\B A \in \dom(\widehat{\B x_0})$. By the above hypothesis we get $\B B \in \dom(\widehat{\B x_0})$, and as $\B A \in \dom(\widehat{\B x_0})$ and $\B B \in \dom(\widehat{x_0})$, we also conclude that $\widehat{x_0}(\B A) =_{\two} \widehat{x_0}(\B B)$
i.e., $x_0 \in B^1$, and hence $A^1 \subseteq B^1$. Working similarly, we get the inclusions $B^1 \subseteq A^1$, $B^0 \subseteq A^0$, and $A^0 \subseteq B^0$.	
\end{proof}

\begin{remark}\label{rem: sep1}
\normalfont
Notice that due to the proof of Proposition~\ref{prp: isswapa1} within $\INT$, the proof of Proposition~\ref{prp: sep1}, as a whole, is within $\INT$ too. Of course, the proof of the separating property alone is within $\MIN$.
\end{remark}

\begin{remark}\label{rem: sep2}
\normalfont
It is the proof of Proposition~\ref{prp: sep1} that forces us to consider Definition~\ref{def: sep}, instead of the simpler formulation with only the last conjunct.
In section~\ref{sec: stonecech} we show that by a Stone-\v{C}ech theorem on swap algebras of type $(\tii)$ every swap algebra of type $(\tii)$ induces a separated swap algebra of type $(\tii)$, hence there is a plethora of separated swap algebras of type $(\tii)$. 
\end{remark}

\begin{proposition}\label{prp: compswap}
If $\B h \in \widehat{(\B A, \B B)}$ and $\B g \in \widehat{(\B B, \B C)}$, then $\B g \circ \B h \in \widehat{(\B A, \B C)}$. Moreover, $\B \id_A := (A, \id_A) \in \widehat{(\B A, \B A)}$. Consequently, the pair $(\SwapAlg_{\tii}, \PartSwapHom)$ of swap algebras of type $(\tii)$ and partial swap-homomorphisms between them is a category\footnote{In~\cite{MWP24} the subcategory of swap algebras of type $(\tii)$ with total swap homomorphisms was considered only.}.
\end{proposition}

\begin{proof}
We show only conditions $(\fii_3)$ and $(\ho_3)$ for $\dom(\B g \circ \B h)$ and $\B g \circ \B h$, respectively, and for the rest we work similarly. Let $a, b \in A$, such that $a \vee b \in \dom(\B g \circ \B h)$ i.e., $a \vee b \in \dom(\B h)$ and $h(a \vee b) =_B h(a) \vee h(b) \in \dom(\B g)$. By condition $(\fii_3)$ on $\dom(\B g)$ we have that $h(a), h(b)  \in \dom(\B g)$ i.e., $a, b \in \dom(\B g \circ \B h)$. Conversely, if $a, b \in \dom(\B g \circ \B h)$, then $a, b \in \dom(\B h)$ and $h(a), h(b) \in \dom(\B g)$. By condition $(\fii_3)$ on $\dom(\B h)$ and $\dom(\B g)$ we have that $a \vee b \in dom(\B h)$ and $h(a) \vee h(b) =_B h(a \vee b)  \in \dom(\B g)$, respectively. Hence,  $a \vee b \in \dom(\B g \circ \B h)$. Condition $(\ho_3)$ for $\B g \circ \B h$ follows from the same condition on $\B g$ and $\B h$.
\end{proof}

\section{Sets with a Boolean inequality}
\label{sec: bool}

In this section we introduce sets with a Boolean inequality, that is sets $X$ with an ``internal'' falsum $\bot_X$. These sets allow a book-keeping of the use of the Ex falso principle $(\EFQ)$ in constructive mathematics. Moreover, if we work with sets with a Boolean inequality, then many implications of the form $\bot_X \To P$ can be shown within minimal logic. The corresponding weak negation $\neg_X P := P \To \bot_X$ associated to a set $X$ with $\bot_X$ has been used successfully already in constructive algebra\footnote{We thank Peter Schuster for pointing~\cite{LQ15} for the use of a local bottom $\bot_R$ and a local negation $\neg_R P := P \to \bot_R$, where $R$ is a ring, in constructive algebra. This attitude is followed also in~\cite{BS24}.} by Lombardi and Quitt\`e in~\cite{LQ15}.
If we restrict to swap algebras with a Boolean inequality, then the proof of the Stone representation theorem for swap algebras of type $(\tii)$ in the next section is within minimal logic.

\begin{definition}\label{def: Booleineq}
A set with a Boolean inequality is a structure $\C X := (X, =_X, \neq_X, 0_X, 1_X)$, such that $(X, =_X)$ is a set and $\neq_X$ is an extensional relation on $X$, such that the following conditions hold:\\[1mm]
$(\bool_1)$ $0_X \neq_X 1_X$.\\[1mm]
$(\bool_2)$ $\forall_{x,y \in X}\big(x =_X y \ \& \ x \neq_X y \To 0_X =_X 1_X\big)$.\\[1mm]
We call $0_X$ and $1_X$ the booleans of $X$. The local top within $X$, the  local bottom\footnote{This is in complete accordance with Bishop's use of a local empty subset $\ \emptys_X$, for each set $X$. This practice is followed 
by Bridges and V\^{\i}\c{t}\u{a} in~\cite{BV06}, but not by Bridges and Richman in~\cite{BR87}.} within $X$, and the local weak negation within $X$ of a formula $P$ in $\BST$ are the following formulas, respectively,
$$\top_X := 0_X \neq 1_X,$$
$$\bot_X := 0_X =_X 1_X,$$
$$\neg_X P := P \To \bot_X.$$
In proofs we use special rules such as the following\footnote{With this rule we avoid the standard Gentzen elimination-rule for disjunction: $P \ \& \ \neg_X P \To Q$ and $Q \To Q$, as that would mean that $\bot_X \To Q$, which is a form of $\EFQ$.} $: \ \big[(P \vee Q) \ \& \ 
\neg_X P\big] \To Q$, which concern the relation between the basic logical connectives with the defined negation $\neg_X P$ of a formula $P$. 
\end{definition}

\begin{example}\label{ex: bool1}
\normalfont
The primitive structure $\textbf{N} := \big(\Nat, =_{\Nat}, \neq_{\Nat}, 0_{\Nat}, 1_{\Nat}\big)$ and the (defined) structure $\B 2 := \big(\D 2, =_{\D 2}, \neq_{\D 2}, 0_{\D 2}, 1_{\D 2}\big)$ are sets with a  Boolean inequality. Conditions $(\bool_1)$ and $(\bool_2)$ hold axiomatically for $\Nat$.
\end{example}

\begin{lemma}\label{lem: botnat}
	\normalfont (i)
	\itshape $\bot_{\Nat} \To \forall_{n \in \Nat}\big(n \neq_{\Nat}
 n\big)$.\\[1mm]
 	\normalfont (ii)
 \itshape $\bot_{\Nat} \To \forall_{n \in \Nat}\big(0_{\Nat} =_{\Nat}
 	n \big)$.\\[1mm]
 	 		\normalfont (iii)
 	\itshape $\bot_{\Nat} \To \forall_{n, m \in \Nat}\big(n =_{\Nat}
 		m\big)$.\\[1mm]
 			\normalfont (iv)
 		\itshape $\bot_{\Nat} \To \forall_{n, m \in \Nat}\big(n \neq_{\Nat}
 		m\big)$.

 \end{lemma}

\begin{proof}
(i) By induction on $n \in \Nat$. First we show that $0 \neq_{\Nat} 0$. As $0_{\Nat} \neq_{\Nat} 1_{\Nat}$, and since by hypothesis $0_{\Nat} =_{\Nat} 1_{\Nat}$, by the extensionality of $\neq_{\Nat}$ we get $0 \neq_{\Nat} 0$. The implication $n \neq_{\Nat} n \To S(n) \neq_{\Nat} S(n)$ follows by the injectivity of the successor function $S \colon \Nat \to \Nat$ i.e., $n \neq_{\Nat} m \To S(n) \neq_{\Nat}  S(m)$, for every $n, m \in \Nat$. To show the latter, we suppose that $S(n) =_{\Nat} S(m)$ and, applying the predecessor function, we get $n =_{\Nat} m$. As $n \neq_{\Nat} m$ by hypothesis, we have that $\bot_{\Nat}$. By the rule on disjunction in Definition~\ref{def: Booleineq} and the decidability of $\neq_{\Nat}$ we get $S(n) \neq_{\Nat} S(m)$.\\
(ii) Again we use induction on $n \in \Nat$. The base case $0_{\Nat} =_{\Nat} 0_{\Nat}$ follows immediately. The implication 
$0_{\Nat} =_{\Nat} 	n \To 0_{\Nat} =_{\Nat} S(n)$ follows from the fact that if $0_{\Nat} =_{\Nat} n$, then $S(0_{\Nat}) =_{\Nat} S(n)$ i.e., $1_{\Nat} =_{\Nat} S(n)$, which by the hypothesis of $\bot_{\Nat}$ it becomes the required $0_{\Nat} =_{\Nat} S(n)$.\\
(iii) It follows immediately by case (ii) and the fundamental properties of $=_{\Nat}$.\\
(iv) It follows immediately by cases (i), (iii) and the given extensionality of $\neq_{\Nat}$.
\end{proof}

\begin{remark}\label{rem: ontheproof}
\normalfont
If we use the properties of multiplication on $\Nat$, the previous proposition can be shown without induction: if $0_{\Nat} =_{\Nat} 1_{\Nat}$, then $0_{\Nat} =_{\Nat}  0_{\Nat} \cdot n = 1_{\Nat} \cdot n =_{\Nat} n$, for every $n \in 
\Nat$, and hence $n =_{\Nat} m$, for every $n, m \in \Nat$. By the primitive extensionality of $\neq_{\Nat}$ the inequality $0_{\Nat} \neq_{\Nat} 1_{\Nat}$, together with the equalities $n =_{\Nat} 0_{\Nat}$ and $1_{\Nat} =_{\Nat} n$ imply that $n \neq_{\Nat} n$, for every $n \in \Nat$. These proofs are extended to an arbitrary ring $R$ with a Boolean inequality.
\end{remark}

\begin{definition}\label{def: complete}
If $\C X := (X, =_X, \neq_X, 0_X, 1_X)$ is a set with a Boolean inequality, we call $\neq_X$ complete, if we can show within $\MIN$ that
$$\bot_X \To \forall_{x \in X}(x \neq_X x).$$
\end{definition}

By Lemma~\ref{lem: botnat}(i) $\textbf{N}$ is a set with a complete Boolean inequality.

\begin{proposition}\label{prp: subnat}
The structures $\big(\C E(\Nat), =_{\C E(\Nat)}, \neq_{\C E(\Nat)}, \emptys_{\Nat}, \Nat\big)$,
$\big(\C E^{\Disj}(\Nat), =_{\C E^{\Disj}(\Nat)}, \neq_{\C E^{\Disj}(\Nat)}, (\emptys_{\Nat}, \Nat), (\Nat, \emptys_{\Nat})\big)$ are sets with a Boolean inequality.
\end{proposition}

\begin{proof}
We first show that $\big(\C E(\Nat), =_{\C E(\Nat)}, \neq_{\C E(\Nat)}, \emptys_{\Nat}, \Nat\big)$ is a set with a Boolean inequality. The extensionality of $\neq_{\C E(\Nat)}$ follows by our remark right after Definition~\ref{def: subineq}. Condition $(\bool_1)$ follows from the fact that $0_{\Nat} \colon \emptys_{\Nat} \neq_{\C E(\Nat)} \Nat$, as $0_{\Nat} \in \Nat \ \& \ 0_{\Nat} \in \emptys_{\Nat}^{\neq_{\Nat}}$. To show that $0_{\Nat} \in \emptys_{\Nat}^{\neq_{\Nat}}$, let $u \in \Nat$ with $u \neq_{\Nat} u$. By decidability $\neq_{\Nat}$ we have that $0_{\Nat} =_{\Nat} u$ or $0_{\Nat} \neq_{\Nat} u$. If $0_{\Nat} =_{\Nat} u$ was the case, then by the extensionallity of $\neq_{\Nat}$ the inequality $u \neq_{\Nat} u$ implies the inequality $0_{\Nat} \neq_{\Nat} 0_{\Nat}$, which together with the equality $0_{\Nat} =_{\Nat} 0_{\Nat}$ implies $\bot_{\Nat}$. Hence, $0_{\Nat} \neq_{\Nat} u$ is the case. Next we show $(\bool_2)$ i.e., if $A, B \in \C E(\Nat)$, then 
$$A =_{\C E(\Nat)} B \ \& \ A \neq_{\C E(\Nat)} B \To \emptys_{\Nat} =_{\C E(\Nat)}  \Nat.$$
By Definition~\ref{def: subineq} suppose that $n \in \Nat$, such that $n \colon A \neq_{\C E(\Nat)} B$, and let, for example, 
$n \in A$ with $n \in B^{\neq_{\Nat}}$. Hence, $n \neq_{\Nat} n$, which together with $n =_{\Nat} n$ imply $\bot_{\Nat}$. The inequality $\emptys_{\Nat} \subseteq \Nat$ follows trivially, as the conclusion $m =_{\Nat} m$ holds, for every $m \in \Nat$. The inclusion $\Nat \subseteq \emptys_{\Nat}$ follows immediately from Lemma~\ref{lem: botnat}(i).

The proof for the structure $\big(\C E^{\Disj}(\Nat), =_{\C E^{\Disj}(\Nat)}, \neq_{\C E^{\Disj}(\Nat)}, (\emptys_{\Nat}, \Nat), (\Nat, \emptys_{\Nat})\big)$ is similar. The extensionality of $ =_{\C E^{\Disj}(\Nat)}$ follows from Proposition~\ref{prp: ext2}(i). As $\Nat$ is inhabited by $0_{\Nat}$, the proof of $0_{\Nat} \colon (\emptys_{\Nat}, \Nat) \neq_{\C E^{\Disj}(\Nat)} (\Nat, \emptys_{\Nat})$ follows immediately by Definition~\ref{def: csineq}. Next we show $(\bool_2)$ i.e., if $\B A, \B B \in \C E^{\Disj}(\Nat)$, then 
$$\B A =_{\C E^{\Disj}(\Nat)} \B B \ \& \ \B A \neq_{\C E^{\Disj}(\Nat)} \B B \To (\emptys_{\Nat}, \Nat)=_{\C E^{\Disj}(\Nat)}  (\Nat, \emptys_{\Nat}).$$
By definition, the equality $(\emptys_{\Nat}, \Nat)=_{\C E^{\Disj}(\Nat)}  (\Nat, \emptys_{\Nat})$ is equivalent to the equality $\emptys_{\Nat} =_{\C E(\Nat)}  \Nat$. By Remark~\ref{rem: equivs}(ii) suppose that $n \in \Nat$, such that $n \colon \B A \neq_{\C E^{\Disj}(\Nat)} \B B$. The two first cases are treated as in the proof of $(\bool_2)$ for the structure $\big(\C E(\Nat), =_{\C E(\Nat)}, \neq_{\C E(\Nat)}, \emptys_{\Nat}, \Nat\big)$. The last case implies that $n \in A^1 \cup A^0$, such that $\chi_{\B A}(n) =_{\D 2} \chi_{\B B}(n)$ and $\chi_{\B A}(n) \neq_{\D 2} \chi_{\B B}(n)$. Hence, we get $\bot_{\D 2}$ i.e., $\bot_{\Nat}$. To get the required equality $\emptys_{\Nat} =_{\C E(\Nat)}  \Nat$ we work exactly as above for the structure $\big(\C E(\Nat), =_{\C E(\Nat)}, \neq_{\C E(\Nat)}, \emptys_{\Nat}, \Nat\big)$. 
\end{proof}

It is easy to see that the inequalities $\neq_{\C E(\Nat)}$ and $ \neq_{\C E^{\Disj}(\Nat)}$ satisfy the completeness property for inhabited subsets and for inhabited complemented subsets, respectively, only. We provide full proofs and a systematic study of sets with a Boolean inequality in~\cite{Pe24c}. The following proposition can also be shown in a straightforward manner.

\begin{proposition}\label{prp: NtoR} The defined structures 
$$\textbf{Z} := \big(\D Z, =_{\D Z}, \neq_{\D Z}, 0_{\D Z}, 1_{\D Z}\big),$$
$$\textbf{Q} := \big(\D Q, =_{\D Q}, \neq_{\D Q}, 0_{\D Q}, 1_{\D Q}\big),$$
$$\textbf{N} \to \textbf{Q} := \big(\D F(\Nat, \D Q), =_{\D F(\Nat, \D Q)}, \neq_{\D F(\Nat, \D Q)}, 0_{\D F(\Nat, \D Q)}, 1_{\D F(\Nat, \D Q)}\big),$$
where $0_{\D F(\Nat, \D Q)}, 1_{\D F(\Nat, \D Q)}$ are the constant functions $0_{\D Q}$ and $1_{\D Q}$, respectively, and
$$\textbf{R} := \big(\Real, =_{\Real}, \neq_{\Real}, 0_{\Real}, 1_{\Real}\big),$$
are sets with a complete Boolean inequality, such that
$$\bot_{\Nat} \TOT \bot_{\D Z} \TOT \bot_{\D Q} \TOT \bot_{\D F(\Nat, \D Q)} \TOT \bot_{\Real} \ \ \ \& \ \ \ \top_{\Nat} \TOT \top_{\D Z} \TOT \top_{\D Q} \TOT \top_{\D F(\Nat, \D Q)} \TOT \top_{\Real}.$$
\end{proposition}

This partially justifies the use of the global bottom $\bot_{\Nat}$ in the formulation of weak negation within $\BISH$. A universal use of $\bot_{\Nat}$ cannot be fully justified since there are set-constructions that do not preserve bottom. For example, if $\B X, \B Y$ are sets with a Boolean inequality, the product structure $(X \times Y, =_{X \times Y}, \neq_{X \times Y}, 0_{X \times Y}, 1_{X \times Y})$, where $0_{X \times Y}, := (0_X, 0_Y)$ and $ 1_{X \times Y} := (1_X, 1_Y)$ is not, in general, a set with a Boolean inequality, as $\bot_{X \times Y} \TOT \bot_X \ \& \ \bot_Y$, and $\bot_X$ alone, or $\bot_Y$ alone, does not imply $\bot_{X \times Y}$. Of course, the product $X \times X$ has the structure of a set with a Boolean inequality.

\begin{definition}\label{def: boolcs}
If $\C X := (X, =_X, \neq_X, 0_X, 1_X)$ is a set with a Boolean inequality, we call $\B A := (A^1, A^0) \in \C E^{\Disj}(X)$ a Boolean complemented subset, if we can show within $\MIN$ that
$$\bot_X \To P^1(x) \ \ \& \ \ \bot_X \To P^0(x),$$
where $P^1(x)$ and $P^0(x)$ are the extensional properties defining $A^1$ and $A^0$, respectively.
\end{definition}

\begin{proposition}\label{prp: bcs1}
Let $\C X := (X, =_X, \neq_X, 0_X, 1_X)$ be a set with a Boolean inequality and $\B A, \B B \in \C E^{\Disj}(X)$.\\[1mm]
\normalfont (i)
\itshape The relation $\neq_X$ is complete if and only if $(X, \emptys_X)$ is Boolean.\\[1mm]
\normalfont (ii)
\itshape If $\B A$ is Boolean, then $- \B A$ is Boolean.\\[1mm]
\normalfont (iii)
\itshape If $\B A$ and $\B B$ are Boolean, then $\B A \vee \B B, \B A \wedge \B B$ and $\B A - \B B$ are Boolean.\\[1mm]
\normalfont (iv)
\itshape If $\B A$ is Boolean, then $A^1 \cap A^0 = \emptys_X$ $($within $\MIN)$.\\[1mm]
\normalfont (v)
\itshape If $\neq_X$ is complete, then the totality of Boolean complemented subsets is a swap algebra of type $(\tii)$.
\end{proposition}

\begin{proof}
(i)-(iii) are straightforward to show.\\
(iv) By Proposition~\ref{prp: intcomp1} it suffices to show within $\MIN$ that $\emptys_X \subseteq A^1 \cap A^0$. If $x \in \emptys_X$, then $x \neq_X x$ and by $(\bool_1)$ we get $\bot_X$. By the definition of a Boolean complemented subset we get immediately $P^1(x) \wedge P^0(x)$ i.e., $x \in A^1 \cap A^0$.\\
(v) It follows immediately from (i)-(iv).
\end{proof}

\section{The representation theorem}
\label{sec: repr}

\begin{theorem}[Stone representation theorem for separated swap algebras of type $(\tii)$ $(\INT)$]\label{thm: stone}
If $A$ is a separated swap algebra of type $(\tii)$, then the assignment routine 
$$\pmb{\Stone} \colon A \sto \C E^{\Disj}(\widehat{\B A}), \ \ \ \ \ a \mapsto \pmb{\Stone}(a),$$
$$\pmb {\Stone}(a) := \big(\Stone^1(a), \Stone^0(a)\big),$$
$$\Stone^1(a) := \{\B f \in \widehat{\B A} \mid a \in \dom(\B f) \wedge f(a) =_{\two} 1\},$$
$$\Stone^0(a) := \{\B g \in \widehat{\B A} \mid a \in \dom(\B g) \wedge g(a) =_{\two} 0\},$$
is a total swap embedding of $A$ into the $($separated$)$ swap algebra $\C E^{\Disj}(\widehat{A})$ of type $(\tii)$. 
\end{theorem}

\begin{proof}
First we show that $\pmb{\Stone}$ is well-defined i.e., $\Stone^1(a) \Disj \Stone^0(a)$ in $(\widehat{\B A}, =_{\widehat{\B A}}, \neq_{\widehat{\B A}})$. If $f \in \Stone^1(a)$ and
$g \in \Stone^0(a)$, then by Definition~\ref{def: equalpf} we have that $a \colon f \neq_{\widehat{\B A}} g$. To show that $\pmb{\Stone}$ is a function, let $a, b \in A$ with $a =_A b$, and we show that $\Stone^1(a) =_{\C E(\widehat{\B A})} \Stone^1(b)$ and $\Stone^0(a) =_{\C E(\widehat{\B A})} \Stone^0(b)$. Both these equalities follow trivially from the extensionality of the domain of a partial function and the fact that a partial function preserves equalities between the elements of its domain. Next we show that $\pmb{\Stone}$ is an embedding. For that let $a, b \in A$, such that 
$\pmb {\Stone}(a) =_{\C E^{\Disj}(\widehat{\B A})} \pmb {\Stone}(b)$, and we show that $a =_A b$. By the separation hypothesis on $A$ it suffices to show that 
\begin{align*}
	& \forall_{\B f \in \widehat{\B A}}\bigg(\big(a \in \dom(\B f) \To b \in \dom(\B f)\big) \wedge  \\
	& \ \ \ \ \ \ \ \ \ \ \ \ \ \ \ \ \big(b \in \dom(\B f) \To a \in \dom(\B f)\big) \wedge \\
	& \ \ \ \ \ \ \ \ \ \ \ \ \ \ \ \ \big(a \in \dom(\B f) \wedge a \in \dom(\B f) \To f(a) =_{\two} f(b)\big)\bigg).
\end{align*}
Let  $\B f \in \widehat{\B A}$. If $a \in \dom(\B f)$, let for example $f(a) =_{\two} 1$ i.e., $\B f \in \Stone^1(a)$. As $\Stone^1(a) =_{\C E(\widehat{\B A})} \Stone^1(b)$, we get $\B f \in 
\Stone^1(b)$ i.e., $b \in \dom(\B f)$ and $f(b) =_{\two} 1$. If $f(a) =_{\two} 0$, we work similarly. The other two conjuncts follow similarly.
Next we show that $\pmb{\Stone}$ is a swap homomorphism. First we show that $\pmb{\Stone}(1) =_{\C E^{\Disj}(\widehat{\B A})} \big(\widehat{\B A}, {\emptys}_{\widehat{\B A}}\big)$. By conditions $(\fii_1)$ and $(\ho_1)$ in Definition~\ref{def: partialshomo} we have that $\Stone^1(1) =_{\C E(\widehat{\B A})} \widehat{\B A}$. By definition we have that
\begin{align*}
\Stone^0(1)	& := \{\B g \in \widehat{\B A} \mid 1 \in \dom(\B g) \wedge g(1) =_{\two} 0\}\\
& :=  \{\B g \in \widehat{\B A} \mid 1 \in \dom(\B g) \wedge 1 =_{\two} 0\}\\ 
& =_{\C E(\widehat{\B A})}  \{\B g \in \widehat{\B A} \mid 1 =_{\two} 0\}.
\end{align*}
By Definition~\ref{def: extempty} we have that $\emptys_{\widehat{\B A}}  := \{\B g \in \widehat{\B A} \mid \B g \neq_{\widehat{\B A}} \B g\}$ and 
\begin{align*}
	\B g \neq_{\widehat{\B A}} \B g & :\TOT \exists_{a \in A}\bigg[\big(a \in \dom(\B g) \wedge a \in \dom(\B g)^{\neq_A}\big) \vee \\
	& \ \ \ \ \ \ \ \ \ \ \ \ \ \big(a \in \dom(\B g) \wedge g(a) \neq_{\two} g(a)\big)\bigg].
\end{align*}
The inclusion $\Stone^0(1) \subseteq 	\emptys_{\widehat{\B A}}$ is shown within $\MIN$ as follows.
Let $\B g \in \widehat{\B A}$ such that $1 =_{\two} 0 =_{\D 2} g(1)$. We show that $1 \colon g \neq_{\widehat{\B A}} g$ by showing that $1 \in \dom(\B g)$ and $g(1) \neq_{\D 2} g(1)$. By Definition~\ref{def: partialshomo} it suffices to show that $1 =_{\D 2} g(1) \neq_{\D 2} g(1)$. By the extensionality of $\neq_{\D 2}$ we have that $g(1) =_{\D 2} 0 \neq_{\D 2} 1 =_{\D 2} g(1)$.
The inclusion $\emptys_{\widehat{\B A}} \subseteq \Stone^0(1)$ follows from\footnote{This is the only point in the proof of Theorem~\ref{thm: stone} that we rely on $\EFQ$. In Proposition~\ref{prp: min} we show how to avoid $\EFQ$ if we assume that $\neq_A$ is a Boolean inequality.} $\EFQ$. 

By  conditions $(\fii_2)$ and $(\ho_2)$ in Definition~\ref{def: partialshomo} we have that $\Stone^1(-a) = \Stone^0(a)$ and $\Stone^0(-a) = \Stone^1(a)$, hence $\pmb {\Stone}(-a) =_{\C E^{\Disj}(\widehat{\B A})} -\pmb {\Stone}(a)$. By definition 
$\pmb {\Stone}(a \vee b) := \big(\Stone^1(a \vee b), \Stone^0(a \vee b)\big)$, where
$$\Stone^1(a \vee b) := \{\B f \in \widehat{\B A} \mid a \vee b \in \dom(\B f) \wedge f(a \vee b) =_{\two} 1\}.$$
$$\Stone^0(a \vee b) := \{\B g \in \widehat{\B A} \mid a \vee b \in \dom(\B g) \wedge g(a \vee b) =_{\two} 0\},$$
Moreover, we have that
$$\pmb {\Stone}(a) \vee \pmb {\Stone}(b) := \bigg(\big[ \Stone^1(a) \cap  \Stone^1(b)\big] \cup \big[ \Stone^1(a) \cap  \Stone^0(b)\big]$$
$$\ \ \ \ \ \ \ \ \ \ \ \ \ \ \ \ \ \ \ \ \ \ \ \ \ \ \ \ \ \ \ \ \ \ \ \cup \big[ \Stone^0(a) \cap  \Stone^1(b)\big], \big[ \Stone^0a) \cap  \Stone^0(b)\big]\bigg),$$
$$\Stone^1(a) \cap  \Stone^1(b) =_{\C E(\widehat{\B A})} \big\{\B f \in \widehat{\B A} \mid a \in \dom(\B f) \wedge b \in \dom(\B f) \wedge f(a) =_{\two} 1 =_{\two} f(b)\big\},$$
$$\Stone^1(a) \cap  \Stone^0(b) =_{\C E(\widehat{\B A})} \big\{\B f \in \widehat{\B A} \mid a \in \dom(\B f) \wedge b \in \dom(\B f) \wedge f(a) =_{\two} 1 \wedge f(b) =_{\two} 0\big\},$$
$$\Stone^0(a) \cap  \Stone^1(b) =_{\C E(\widehat{\B A})} \big\{\B f \in \widehat{\B A} \mid a \in \dom(\B f) \wedge b \in \dom(\B f) \wedge f(a) =_{\two} 0 \wedge f(b) =_{\two} 1\big\},$$
$$\Stone^0(a) \cap  \Stone^0(b) =_{\C E(\widehat{\B A})} \big\{\B f \in \widehat{\B A} \mid a \in \dom(\B f) \wedge b \in \dom(\B f) \wedge f(a) =_{\two} 0 =_{\two} f(b)\big\}.$$
Next we show the required equalities
$$\Stone^1(a \vee b) =_{\C E(\widehat{\B A})} \big[ \Stone^1(a) \cap  \Stone^1(b)\big] \cup \big[ \Stone^1(a) \cap  \Stone^0(b)\big] \cup 
 \big[ \Stone^0(a) \cap  \Stone^1(b)\big],$$
 and
 $$\Stone^0(a \vee b) =_{\C E(\widehat{\B A})} \Stone^0(a) \cap  \Stone^0(b).$$
 For the first equality let $f \in \widehat{\B A}$, such that $a \vee b \in \dom(\B f) \wedge f(a \vee b) =_{\two} 1$. By conditions $(\fii_3)$ and $(\ho_3)$ in Definition~\ref{def: partialshomo} we have that $a \in \dom(\B f), b \in \dom(\B f)$, and $f(a \vee b) =_{\two} f(a) \vee f(b) =_{\two} 1$, hence $f(a) =_{\two} 1$ or 
 $f(b) =_{\two} 1$ 
 i.e., $f \in \big[\Stone^1(a) \cap  \Stone^1(b)\big] \cup \big[ \Stone^1(a) \cap  \Stone^0(b)\big] \cup \big[ \Stone^0(a) \cap  \Stone^1(b)\big]$. For the converse inclusion we work similarly. For the second equality let $f \in \widehat{\B A}$, such that $a \vee b \in \dom(\B f) \wedge f(a \vee b) =_{\two} 0$. By conditions $(\fii_3)$ and $(\ho_3)$ in Definition~\ref{def: partialshomo} again we have that $a \in \dom(\B f), b \in \dom(\B f)$, and $f(a \vee b) =_{\two} f(a) \vee f(b) =_{\two} 0$, hence $f(a) =_{\two} 0$ and $f(b) =_{\two} 0$ i.e., $f \in \big[\Stone^0(a) \cap  \Stone^0(b)\big]$.  For the converse inclusion we work similarly. 
\end{proof}

The definition of a field of elements in a swap algebra $A$ and of a swap character on $A$ are tailored to the previous proof. Notice that f we had considered a swap algebra of type $(\ti)$, then the equality $\B {\Stone}(a \vee b) = \B {\Stone} (a) \cup \B {\Stone}(b)$ fails i.e., the above proof of the Stone representation works only for swap algebras of type $(\tii)$. The reason is that 
$$\B {\Stone}(a) \cup \B {\Stone}(b) := \bigg(\big[\Stone^1(a) \cup \Stone^1(b), \Stone^0(a) \cap \Stone^0(b)\bigg),$$
and while $\Stone^1(a \vee b) \subseteq \Stone^1(a) \cup \Stone^1(b)$, the converse inclusion does not hold in general, due to the partiality of the swap characters. If $f$ is a swap character on $A$ with $a \in \dom(\B f)$ and $f(a) =_{\two} 1$, we don't have in general that $a \vee b \in \dom(\B f)$ i.e., that $b \in \dom(\B f)$.

\begin{proposition}\label{prp: min}
	If  $\neq_A$ is a Boolean inequality, then the proof of Theorem~\ref{thm: stone} is within $\MIN$.
\end{proposition}

\begin{proof}
It suffices to show that in this case the inclusion $\emptys_{\widehat{\B A}} \subseteq \Stone^0(1)$ is shown within $\MIN$. For that, let $\B g \in \emptys_{\widehat{\B A}} :\TOT \B g \neq_{\B {\widehat{A}}} \B g$ i.e., there is $a \in A$ such that either $a \in \dom(\B g) \ \& \ a \in  \dom(\B g)^{\neq_A}$ or $a \in \dom(\B g) \ \& \ g(a) \neq_{\D 2} g(a)$. In the latter case let $g(a) =_{\D 2} 1$. Hence $g(a) \neq_{\D 2} 1$ i.e., $g(a) =_{\D 2} 0$. Consequently, $1 =_{\D2} 0$. If $g(a) =_{\D 2} 0$, we work similarly. Suppose next that $a \in \dom(\B g) \ \& \ a \in  \dom(\B g)^{\neq_A}$, hence $a \neq_A a$. By condition $(\bool_1)$ we get $(0_A =_A 1_A)$, therefore $0 =_{\D 2} g(0_A) =_{\D 2} g(1_A) =_{\D 2} 1$.
\end{proof}

In the next section we explain why the hypothesis of $A$ being separated in Theorem~\ref{thm: stone} is not a loss of generality from the point of view of the theory of swap characters.

\section{A Stone-\v{C}ech theorem for swap algebras of type  $(\tii)$}
\label{sec: stonecech}

As we mentioned already, classically every Boolean algebra is separated, which amounts to showing the contraposition of the separation of the points of $A$ from its characters $\widehat{\B A}$ i.e., the existence for every non-zero element $a$ of $A$, of a swap character on $A$ that maps $a$ to $1$. The classical proof of this fact employs Zorn's lemma to justify the existence of an appropriate maximal ideal (see~\cite{Ha74}, p.~77). What we show in this section is that the separation property of $A$ can be handled in a different, and much easier way. The main idea is the recognition that we do not need to employ some maximal (or prime) ideal, rather to change the equality of $A$. In topology a similar attitude is followed in the theory of the ring of continuous functions $C(X)$ of a topological space~\cite{GJ60}. Within this theory it is not necessary to work with arbitrary topological spaces, but only with completely regular ones. In a completely regular topological space $(X, \mathcal{T})$ a pair $(x, B)$, where
$B$ is closed and $x \notin B$, is separated by some $f \in C(X, [0, 1])$. The ring of real-valued,
continuous functions $C(X)$ of a completely regular and $T_{1}$-space $X$,
also known as a \textit{Tychonoff space}, separates the points of $X$ i.e.,
$$\forall_{x, x{'} \in X}\big(\forall_{f \in C(X)}(f(x) = f(x{'})) \Rightarrow x = x{'}\big).$$
The sufficiency of the completely regular topological spaces in the theory of $C(X)$
is provided by the Stone-\v{C}ech theorem, according to which, for every topological 
space $X$ there exists a completely regular space $\rho X$ and a continuous mapping $\rho_X \colon X \rightarrow \rho X$ such that the induced function $f \mapsto \rho^{*}_X(f)$, where
$\rho^{*}_X(f) = f \circ \rho_X$, is a ring isomorphism between $C(\rho X)$ and $C(X)$ (see~\cite{GJ60}, p.~41).
\begin{center}
	\begin{tikzpicture}
		
		\node (E) at (0,0) {$X$};
		\node[right=of E] (F) {$\rho X$};
		\node [below=of F] (D) {$\Real$};
		\node[right=of F] (P) {};
		\node[right=of P] (S) {};
		\node[right=of S] (K) {$X$};
		\node [right=of K] (L) {$\rho X$};
		\node [below=of L] (M) {$Y$};
		
		\draw[->] (E)--(F) node [midway,above] {$ \rho_X$};
		\draw[->] (E)--(D) node [midway,left] {$C(X) \ni \rho^{*}_X(f) \ \ $};
		\draw[->] (F)--(D) node [midway,right] {$f \in C(\rho X) $};
		\draw[->] (K)--(M) node [midway,left] {$C(X, Y) \ni h \ \ $};
		\draw[->] (K)--(L) node [midway,above] {$\rho_X$};
		\draw[->] (L)--(M) node [midway,right] {$ \rho h \in C(\rho X, Y)$};

	\end{tikzpicture}
\end{center} 
Consequently, a functor $\rho \colon \Top \to \crTop$ from the category of topological spaces $(\Top, \Cont)$ to its subcategory of completely regular topological spaces $(\crTop, \Cont)$ is defined, which is a reflector, that is for every continuous function $h$ from $X$ to a completely regular space $Y$ there is a unique continuous function $\rho h \colon \rho X \to Y$ such that the above right triangle commutes (see~\cite{He68}, p.~6). In complete analogy, here we show an appropriate Stone-\v{C}ech theorem for swap algebras of type $(\tii)$, hence also for Boolean algebras, which implies that from the point of view of the theory of swap characters $\widehat{\B A}$, our current analogue to $C(X)$, it suffices to work with separated swap (or Boolean) algebras. Although no topological notions are involved in the formulation and the proof of Theorem~\ref{thm: sc}, the structure of our constructive proof of this theorem is very similar to the structure of the classical proof of the topological Stone-\v{C}ech theorem (see~\cite{Wa74}, p.~5--6) . This is the reason why we chose to name it like that. The similarity is even more evident, if one compares the proof of Theorem~\ref{thm: sc} with the constructive proof of Stone-\v{C}ech theorem for Bishop spaces (see~\cite{Pe15a}, p.~307).

\begin{theorem}[Stone-\v{C}ech theorem for swap algebras of type $(\tii)$]\label{thm: sc} 
	If $(A, =_A, \neq_A)$ is a swap algebra of type $(\tii)$, there is a swap algebra of type $(\tii)$
	 $(\sigma A, =_{\sigma A}, \neq_{\sigma A})$ and a total swap homomorphism $\sigma_A \colon A \to 
	\sigma A$, such that the following hold:\\[1mm]
	\normalfont (i)
	\itshape For every $\B f := (\dom(\B f), f) \in \widehat{\B A}$ there is a unique $\B {\sigma f} := (\dom(\B {\sigma f}), \sigma f) \in \widehat{\B {\sigma A}}$, such that the following left triangle commutes\footnote{For simplicity, in the diagrams we write $\sigma f \in \widehat{\B {\sigma A}}$, instead of $\B {\sigma f} \in \widehat{\B {\sigma A}}$.} i.e.,
	$\dom(\B {\sigma f}) := \{\sigma_A(a) \mid a \in \dom(\B f)\}$ and $\sigma f(\sigma_A(a)) =_{\two} f(a)$, for every $a \in \dom(\B f)$.	
	\begin{center}
		\begin{tikzpicture}
			
			\node (E) at (0,0) {$A$};
			\node[right=of E] (F) {$\sigma A$};
			\node [below=of F] (D) {$\two$};
			\node[right=of F] (P) {};
			\node[right=of P] (N) {};
			\node[right=of N] (S) {};
			\node[right=of S] (K) {$A$};
			\node [right=of K] (L) {$\sigma A$};
			\node [below=of L] (M) {$B$};
			
			\draw[->] (E)--(F) node [midway,above] {$ \sigma_A$};
			\draw[-right to] (E) to node [midway,left] {$\widehat{\B A} \ni f \ \ $} (D) ;
			\draw[-left to] (F) to node [midway,right] {$\sigma f \in \widehat{\B {\sigma A}}$} (D);
			\draw[-right to] (K) to node [midway,left] {$\widehat{(\B A,\B B)} \ni h \ \ $} (M);
			\draw[->] (K)--(L) node [midway,above] {$\sigma_A$};
			\draw[-left to] (L) to node [midway,right] {$ \sigma h \in \widehat{(\B {\sigma A}, \B B)}$} (M);

		\end{tikzpicture}
	\end{center} 
	\normalfont (ii)
	\itshape For every $\B h \in \widehat{\B {\sigma A}}$ there is a unique $\B {h{'}} \in \widehat{\B A}$, such that $\B {\sigma h{'}} =_{\widehat{\B {\sigma A}}} \B h$, and hence 
		$\widehat{\B A} =_{\D V_0} \widehat{\B {\sigma A}}$.\\[1mm]	
	\normalfont (iii)
	\itshape The swap algebra $\sigma A$ is separated.\\[1mm]	
	\normalfont (iv)
	\itshape If $(A, =_A, \neq_A, 0_A, 1_A)$ is a set with a Boolean inequality, then $(\sigma A, =_{\sigma A}, \neq_{\sigma A}, 0_{\sigma A}, 1_{\sigma A})$ is a set with a Boolean inequality.\\[1mm]
	\normalfont (v)
	\itshape
	If $(B, =_B, \neq_B)$ is a separated swap algebra of type $(\tii)$, and $h \colon A \pto B$ is a partial swap homomorphism,
	 there is a unique partial swap homomorphism $\sigma h \colon \sigma A \pto B$, such that the above right triangle commutes i.e., 
	 $\dom(\B {\sigma h}) := \{\sigma_A(a) \mid a \in \dom(\B h)\}$ and $\sigma h(\sigma_A(a)) =_{\two} h(a)$, for every $a \in \dom(\B h)$.\\[1mm]
	 	\normalfont (vi)
	 \itshape For every $\B  h \in \widehat{(\B {\sigma A}, \B B)}$ there is a unique $\B h{'} \in \widehat{(\B A, \B B)}$ with $\B {\sigma h{'}} =_{\widehat{(\B {\sigma A}, \B B)}} \B h$, and  $\widehat{(\B A, \B B)} =_{\D V_0} 
	 \widehat{(\B {\sigma A}, \B B)}$.
		\end{theorem}

\begin{proof}
If $\sigma A := A$, it is straightforward to show that the following relation is an equivalence relation 
\begin{align*}
	a =_{\sigma A} b &:\TOT \forall_{\B f \in \widehat{\B A}}\big((a \in \dom(\B f) \TOT b \in \dom(\B f)) \wedge\\
	& \ \ \ \ \ \ \ \ \ \ \ \ \  (a \in \dom(\B f) \wedge b \in \dom(\B f) \To f(a) =_{\two} f(b))\big).
\end{align*}
As an inequality on $\sigma A$ we consider the following relation
\begin{align*}
	a \neq_{\sigma A} b &:\TOT \exists_{\B f \in \widehat{\B A}}\bigg(\big(a \in \dom(\B f) \wedge b \in \dom(\B f)^{\neq_{A}}\big) \vee\\
		& \ \ \ \ \ \ \ \ \ \ \ \ \ \big(b \in \dom(\B f) \wedge a \in \dom(\B f)^{\neq_{A}}\big) \vee\\
	& \ \ \ \ \ \ \ \ \ \ \ \ \  \big(a \in \dom(\B f) \wedge b \in \dom(\B f) \wedge f(a) \neq_{\two} f(b)\big)\bigg).
\end{align*}
The extensionality of $a \neq_{\sigma A} b$ follows from Proposition~\ref{prp: ineqext}, hence 
$(\sigma A, =_{\sigma A}, \neq_{\sigma A}) \in \SetExtIneq$. 
Next we show that $\sigma A$ is a swap algebra of type $(\tii)$. All operations on $\sigma A$ are defined as in $A$. For example, $\vee \colon \sigma A \times \sigma A \to \sigma A$ is defined by the rule $(a, b) \mapsto a \vee b$. To show that this is a function, let $a =_{\sigma A} a{'},$ $b =_{\sigma A} b{'}$, and we show that $a \vee b =_{\sigma A} a{'} \vee b{'}$. 
If $\B f \in \widehat{\B A}$ with $a \vee b \in \dom(\B f)$, then by $(\fii_3)$ we have that
 $a \in \dom(\B f)$ and $b \in \dom(\B f)$. By the equalities $a =_{\sigma A} a{'}$ and $b =_{\sigma A} b{'}$ we get  $a{'} \in \dom(\B f)$ and $b{'} \in \dom(\B f)$,
 hence $a{'} \vee b{'} \in \dom(\B f)$. For the converse implication we work similarly. If $a \vee b \in \dom(\B f)$ and $a{'} \vee b{'} \in \dom(\B f)$, then 
 $a, a{'} \in \dom(\B f)$ and $b, b{'} \in \dom(\B f)$, and by the supposed equalities we have that $f(a) =_{\two} f(a{'})$ and $f(b) =_{\two} f(b{'})$. Consequently,
 $f(a \vee b) =_{\two} f(a) \vee f(b) =_{\two} f(a{'}) \vee f(b{'}) =_{\two} f(a{'} \vee b{'})$. For the remaining swap operations we work similarly. All axioms of a swap algebra of type $(\tii)$ hold for $\sigma A$ from the hypothesis that $A$ is a swap algebra of type $(\tii)$, and the obvious implication
 $$(*) \ \ \ \ \ \ \ \ \ \ \ \ \ \ \ \ \ \ \ \  \ \ \ \ \ \ \ \ \ \ \ \  \   \ \ \ \ \ \ \ \ \ \ \ \  \   \ \ \ \ \ \ \ \ \ \ \ \  \ \ \ \ \ \ \ \ \   \ a =_A b \To a =_{\sigma A} b, \ \ \ \ \ \ \ \ \ \ \ \  \   \ \ \ \ \ \ \ \ \ \ \ \  \ 
  \ \ \ \ \ \ \ \ \ \ \ \  \   \ \ \ \ \ \ \ \ \ \ \ \  \   \ \ \ \ \ \ \ \ \ \ \ \  \   \ \ \ \ \ \ \ \ \ \ \ \  \   \ \ \ \ \ \ \ \ \ \ \ \  \ $$
 for every $a, b \in A$. The fact that $\sigma$ is a swap-homomorphism follows trivially by its definition and the definition of the swap operations on $\sigma A$.  \\
 (i) If $\B f := (\dom(\B f), f) \in \widehat{\B A}$, let the assignment routine $\sigma f \colon \dom(f) \sto \two$, defined by the rule $\sigma f(a) := f(a)$, for every $a \in \dom(\B f)$. First we show that $\sigma f$ is a partial function from $\sigma A$ to $\two$. For that, it suffices to show that it preserves the corresponding equalities. Let $a, b \in \dom(\B f)$, such that $a =_{\sigma A} b$. By definition of the equality $=_{\sigma A}$ this implies that $f(a) =_{\two} f(b)$ i.e., $\sigma f(a) =_{\two} \sigma f(b)$. Next we show that $\B {\sigma f} := (\dom(\B f), \sigma f) \in \widehat{\B {\sigma A}}$. The defining conditions of a swap character follow immediately from the definition of $\sigma f$ and of the swap operations in $\sigma A$. Clearly, $\sigma f$ makes the above left triangle commutative. The uniqueness of $\sigma f$ in that commutativity is immediate to show.\\
 (ii) Let $h{'} \colon \dom(\B h) \sto \two$ be defined by the rule $h{'}(a) := h(a)$, for every $a \in \dom(\B h)$. By implication $(*)$ above we have that $h{'}$ is a function, and by the definition of the swap operations on $\sigma A$ we have that $h{'} \in \widehat{\B A}$. The equality $\B {\sigma h{'}} =_{\widehat{\B {\sigma A}}} \B h$ follows from case (i), and the uniqueness of $h{'}$ with respect to this property is imediate to show.\\
 (iii)  To prove that $\sigma A$ is separated, let 
 $$\forall_{\B h \in \widehat{\B {\sigma A}}}\bigg((a \in \dom(\B h) \TOT b \in \dom(\B h)) \wedge  \big(a \in \dom(\B h) \wedge b \in \dom(\B h) \To h(a) =_{\two} h(b)\big)\bigg),$$
 which implies, actually by (ii) it is equivalent to,
 $$\forall_{\B f \in \widehat{\B A}}\bigg((a \in \dom(\sigma \B f) \TOT b \in \dom(\B {\sigma f})) \wedge  \big(a \in \dom(\B {\sigma f}) \wedge b \in \dom(\B {\sigma f}) \To \sigma f(a) =_{\two} \sigma f(b)\big)\bigg),$$
 which in turn by definition of $\sigma f$ is equivalent to the equality $a =_{\sigma A} b$.\\
 (iv) Let $0_{\sigma A} := 0_A$ and $1_{\sigma A} := 1_A$. By implication $(*)$ we have that $(\bool_1)$ for $A$ implies $(\bool_1)$ for $\sigma A$. Next we show that $(\bool_2)$ for $A$ implies $(\bool_2)$ for $\sigma A$. If $a, b \in A$ with $a =_{\sigma A} b$ and $a \neq_{\sigma A} b$, we show that $0_A =_{\sigma A} 1_A$, where by definition we have that
 \begin{align*}
 \bot_{\sigma A} & := 0_A =_{\sigma A} 1_A\\
 & :\TOT \forall_{\B f \in \B {\widehat{A}}}\bigg(\big[0_A \in \dom(\B f) \TOT 1_A \in \dom(\B f)\big] \ \& \ \big[0_A \in \dom(\B f) \ \& \ 1_A \in \dom(\B f) \To f(0_A) =_{\two} f(1_A)\big]\bigg)\\
 & \TOT \forall_{\B f \in \B {\widehat{A}}}\big(f(0_A) =_{\two} f(1_A)\big)\\
 & \TOT 0_{\two} =_{\two} 1_{\two}\\
 & \TOT: \bot_{\two}.
 \end{align*}
 Let $\B f \in \B {\widehat{A}}$, such that $\B f : a \neq_{\sigma A} b$. If $a \in \dom(\B f)$ and $b \in \dom(\B f)^{\neq_A}$, then $b \neq_A b$, since by the hypothesis $a =_{\sigma A} b$ we have that $a \in \dom(\B f) \TOT b \in \dom(\B f)$. As $b =_A b$, by $(\bool_2)$ for $A$ we get $0_A =_A 1_A$, hence by $(*)$ we het $0_{A} =_{\sigma A} 1_A$. If $b \in \dom(\B f)$ and $a \in \dom(\B f)^{\neq_A}$,
 we work as in the previous case. Next we suppose that $a, b \in \dom(\B f)$ and $f(a) \neq_{\two} f(b)$.
 Since the hypothesis $a =_{\sigma A} b$ implies in this case that $f(a) =_{\two} f(b)$, condition $(\bool_2)$ for $\two$ implies $\bot_{\two}$, and hence $\bot_{\sigma A}$.\\
 (v) If $\B h := (\dom(\B h), h) \in \widehat{(\B A,\B B)}$, let the assignment routine $\sigma h \colon \dom(h) \sto B$, defined by the rule $\sigma h(a) := h(a)$, for every $a \in \dom(\B h)$. First we show that $\sigma h$ is a partial function from $\sigma A$ to $B$. For that, it suffices to show that it preserves the corresponding equalities. Let $a, b \in \dom(\B h)$, such that $a =_{\sigma A} b$. By the separation hypothesis on $B$, to show that $h(a) =_{\two} h(b)$, it suffices to show that
 $$\forall_{\B g \in \widehat{\B B}}\bigg((h(a) \in \dom(\B g) \TOT h(b) \in \dom(\B g)) \wedge  \big(h(a) \in \dom(\B g) \wedge h(b) \in \dom(\B g) \To g(h(a)) =_{\two} g(h(b))\big)\bigg).$$
 If $\B g \in \widehat{\B B}$, then by Proposition~\ref{prp: compswap} we have that $\B g \circ \B h \in \widehat{\B A}$. Then the formulas $(h(a) \in \dom(\B g) \TOT h(b) \in \dom(\B g))$ and $\big(h(a) \in \dom(\B g) \wedge h(b) \in \dom(\B g) \To g(h(a)) =_{\two} g(h(b))$ follow immediately by the definition $a =_{\sigma A} b$ for 
  $\B g \circ \B h \in \widehat{\B A}$. For the remaining properties of $\sigma h$ we work as in the proof of (i).\\
  (vi) We work similarly to the proof of case (ii).
 \end{proof}

Notice that in the proofs of conditions (i), (iii), and (v) of Theorem~\ref{thm: sc} we avoided the converse to the implication $(*)$, which in general doesn't hold, by using the definition or the property of separation. Let the functor $\sigma \colon \SwapAlg_{\tii} \to \SepSwapAlg_{\tii}$, defined by
the rules $A \mapsto \sigma A$ and $(\widehat{(\B A,\B B)} \ni \B h \mapsto \B {\sigma(\sigma_B \circ h)} \in \widehat{(\B {\sigma A}, \B {\sigma B})}$
where, according to Theorem~\ref{thm: sc}$($v$)$, $ \B {\sigma(\sigma_B \circ h)}$ is the unique partial function that makes the following rectangle commutative in the obvious sense
\begin{center}
	\begin{tikzpicture}
		
		\node (E) at (0,0) {$A \ $};
		\node[right=of E] (K) {};
		\node[right=of K] (F) {$\sigma A$};
		\node[below=of E] (B) {$B$};
		\node[right=of B] (C) {};
		\node[right=of C] (A) {$\sigma B$.};

		\draw[->] (E)--(F) node [midway,above] {$\sigma_A$};
		\draw[-right to] (E) to node [midway,left] {$ h $} (B);
		\draw[-right to] (F) to node [midway,right] {$\sigma(\sigma_B \circ h)$} (A);
		\draw[->] (B)--(A) node [midway,below] {$\sigma_B$};
		\draw[-right to] (E) to node [midway,above] {} (A);

	\end{tikzpicture}
\end{center}  
The fact that $\sigma$ is a functor follows in a straightforward way. If $\Emb \colon \SepSwapAlg_{\tii} \to \SwapAlg$ is the corresponding embedding functor, then $\sigma$ is left adjoint to $\Emb$, since By Theorem~\ref{thm: sc} it follows that $\SepSwapAlg_{\tii}$ is reflective in $\SwapAlg_{\tii}$, which is equivalent to $\sigma$ being left adjoint to $\Emb$.

\section{The special case of Boolean algebras}
\label{sec: ba}

In this section we discuss the effect of our previous results for swap algebras of type $(\tii)$ on the constructive representation theory of Boolean algebras. As we have already explained earlier, a Boolean algebra is a swap agebra of type $(\tii)$ (and of type $(\ti)$) in a trivial way, since in a Boolean algebra
$1_a =_A 1$ and $0_a =_A 0$, for every $a \in A$. As a consequence, a field of elements $F$ in $A$ is always $A$ itself, as by the equality $1 =_A a \vee (-a)$ and condition $(\fii_3)$ we get $a \in F$. Consequently, a swap character on a Boolean Algebra $A$ is always total\footnote{For this reason, we avoid using bold symbols for swap characters on a Boolean algebra $A$ and for their set $\widehat{A}$.}, and the separation condition on $A$ takes the expected form
$$\forall_{a, b \in A}\bigg(\forall_{f \in \widehat{A}}\big(f(a) =_{\two} f(b)\big) \To a =_A b\bigg).$$
If we restrict to one-dimensional subsets, then the classical theory of Boolean algebras and Boolean-homomorphisms cannot be constructivised in a direct way. E.g., even to show that the powerset $\C P(X)$ of a set $X$ is a separated Boolean algebra, we need to use the pointed Boolean homomorphisms
$$\widehat{x_0}(A) := \left\{ \begin{array}{ll}
	1   &\mbox{, $x_0 \in A$}\\
	0             &\mbox{, $x_0 \notin A$}
\end{array}
\right. $$
which depend on $\PEM$. We bypass this problem by considering total complemented subsets $\B A := (A^1, A^0)$, with 
$A^1 \cup A^0 =_{\C E(X)} X$, and then define the pointed Boolean homomorphisms by
	$$\widehat{x_0}(\B A) := \left\{ \begin{array}{ll}
	1   &\mbox{, $x_0 \in A^1$}\\
	0             &\mbox{, $x_0 \in A^0$.}
\end{array}
\right. $$
Working as in the proof of Proposition~\ref{prp: sep1} we get that the total elements $\Tot^{\Disj}(X)$ of $\C E^{\Disj}(X)$ is a separated Boolean algebra. Our Stone representation theorem for separated swap algebras of type $(\tii)$ takes the following form for separated Boolean algebras.

\begin{theorem}[Constructive Stone representation theorem for separated Boolean algebras]\label{thm: stoneba}
	If $A$ is a separated Boolean algebra, then the function 
	$\pmb{\Stone} \colon A \sto \Tot\big[\C E^{\Disj}(\widehat{A})\big]$ where $a \mapsto \pmb{\Stone}(a),$ and
	$\pmb {\Stone}(a) := \big(\Stone^1(a), \Stone^0(a)\big),$ with
	$\Stone^1(a) := \{f \in \widehat{A} \mid f(a) =_{\two} 1\}$, and $\Stone^0(a) := \{g \in \widehat{A} \mid g(a) =_{\two} 0\},$
	is a total Boolean-embedding of $A$ into the Boolean algebra $\Tot\big[\C E^{\Disj}(\widehat{A})\big]$ of total elements of $\C E^{\Disj}(\widehat{A})$.
\end{theorem}

\begin{proof}
It follows immediately by Theorem~\ref{thm: stone} and the obvious fact that $\pmb {\Stone}(a)$ is a total complemented subset of $X$, for every $a \in A$.
\end{proof}

The restriction to separated Boolean algebras in the formulation of our constructive Stone representation theorem for separated Boolean algebras is not a real loss of generality, as the Stone-\v{C}ech theorem for swap algebras of type $(\tii)$ can also be restricted to Boolean algebras. The next theorem follows immediately by Theorem~\ref{thm: sc}.

\begin{theorem}[Stone-\v{C}ech theorem for Boolean algebras]\label{thm: scba} 
	If $(A, =_A, \neq_A)$ is a Boolean algebra, there is a Boolean algebra 
	$(\sigma A, =_{\sigma A}, \neq_{\sigma A})$ and a total Boolean-homomorphism $\sigma_A \colon A \to 
	\sigma A$, such that the following hold:\\[1mm]
	\normalfont (i)
	\itshape For every $f \in \widehat{A}$ there is a unique $\sigma f \in \widehat{\sigma A}$, such that the following left triangle commutes.
	\begin{center}
		\begin{tikzpicture}
			
			\node (E) at (0,0) {$A$};
			\node[right=of E] (F) {$\sigma A$};
			\node [below=of F] (D) {$\two$};
			\node[right=of F] (P) {};
			\node[right=of P] (N) {};
			\node[right=of N] (S) {};
			\node[right=of S] (K) {$A$};
			\node [right=of K] (L) {$\sigma A$};
			\node [below=of L] (M) {$B$};
			
			\draw[->] (E)--(F) node [midway,above] {$ \sigma_A$};
			\draw[->] (E) to node [midway,left] {$\widehat{A} \ni f \ \ $} (D) ;
			\draw[->] (F) to node [midway,right] {$\sigma f \in \widehat{\sigma A}$} (D);
			\draw[->] (K) to node [midway,left] {$\widehat{(A,B)} \ni h \ \ $} (M);
			\draw[->] (K)--(L) node [midway,above] {$\sigma_A$};
			\draw[->] (L) to node [midway,right] {$ \sigma h \in \widehat{(\sigma A,B)}$} (M);

		\end{tikzpicture}
	\end{center} 
	\normalfont (ii)
	\itshape For every $h \in \widehat{\sigma A}$ there is a unique $h{'} \in \widehat{A}$, such that $\sigma h{'} =_{\widehat{\sigma A}} h$, and hence 
	$\widehat{A} =_{\D V_0} \widehat{\sigma A}$.\\[1mm]	
	\normalfont (iii)
	\itshape The Boolean algebra $\sigma A$ is separated.\\[1mm]	
	\normalfont (iv)
	\itshape If $(A, =_A, \neq_A, 0_A, 1_A)$ is a set with a Boolean inequality, then $(\sigma A, =_{\sigma A}, \neq_{\sigma A}, 0_{\sigma A}, 1_{\sigma A})$ is a set with a Boolean inequality.\\[1mm]
	\normalfont (v)
	\itshape
	If $(B, =_B, \neq_B)$ is a separated Boolean algebra, and $h \colon A \to B$ is a Boolean-homomorphism,
	there is a unique Boolean-homomorphism $\sigma h \colon \sigma A \to B$, such that the above right triangle commutes.\\[1mm]
	\normalfont (vi)
	\itshape For every $h \in \widehat{(\sigma A, B)}$ there is a unique $h{'} \in \widehat{(A, B)}$ with $\sigma h{'} =_{\widehat{(\sigma A, B)}} h$, and $\widehat{(A, B)} =_{\D V_0} \widehat{(\sigma A, B)}$.
\end{theorem}

\section{Concluding comments}
\label{sec: concl}

In~\cite{MWP24b} we show classically a Stone representation theorem for arbitrary swap rings.
In this work we prove intuitionistically a Stone representation theorem for separated swap algebras of type $(\tii)$, which induces a Stone representation theorem for separated Boolean algebras. We also introduce sets with a Boolean inequality, in order to avoid an unnecessary use of $\EFQ$. If we restrict to separated swap algebras with a Boolean inequality, then their Stone representation is shown within minimal logic. Our constructive approach to the Stone representation theorem for separated swap algebras of type $(\tii)$, and as a special case to the Stone representation theorem for separated Boolean algebras, reveals a new picture. Not only we avoid the Boolean Prime Ideal Theorem in the proof of the classical Stone representation theorem, a slightly weaker version of the axiom of choice, but we also show that by focusing on separated swap algebras of type $(\tii)$ (separated Boolean algebras) and by proving a Stone-\v{C}ech theorem for swap algebras of type $(\tii)$ (a Stone-\v{C}ech theorem for Boolean algebras) the Stone representation of these structures is more or less ``a problem of equality and inequality'', rather than a problem of existence of certain prime or maximal ideals. Because of this, the main results presented here provide a new insight also to the classical theory of Boolean algebras, an insight that comes from constructive mathematics\footnote{Probably, this insight couldn't be found if we were restricted to the classical and one-dimensional approach to sets and subsets.}.

\begin{table}
\caption{Correspondence between Boolean algebras and swap algebras of type $(\tii)$}
\vspace{3mm}
\label{catComp}
\begin{tabular}{l@{\hspace{2.5cm}}l}
	\itshape Boolean algebras  & \itshape Swap algebras of type $(\tii)$\\
	\hline
	Boolean algebra $A$ &     swap algebra of type $(\tii)$ $A$\\
	(total) Boolean character on $A$ & (partial) swap character on $A$\\
	pointed Boolean character on $A$ & pointed swap character on $A$\\
	separated $A$ &  separated $A$\\
	 Stone representation for $A$ & Stone representation for $A$\\
	Stone-\v{C}ech theorem for $A$ & Stone-\v{C}ech theorem for $A$\\\hline
\end{tabular}
\end{table}

It is worth exploring whether theorems of Stone-\v{C}ech type can be shown in the representation theory of other structures. A first candidate-structure is that of a commutative $C^*$-algebra, since the Gelfand representation of a commutative $C^*$-algebra is a generalisation of the Stone representation of a Boolean algebra. It is expected that a similar partial treatment of the notion of a character will be crucial to such a constructive treatment of the Gelfand representation\footnote{A constructive and ``total'' approach to $C^*$-algebras is elaborated in~\cite{Ta05}. See also the constructive proof of a localic Gelfand duality by Coquand and Spitters in~\cite{CS09}.}. The real version of the partial Gelfand transform introduced in section~\ref{sec: gelfand} is a very first step towards this direction, since the formulation and the proof of Proposition~\ref{prp: pGelfand} go through also for the real partial Gelfand transform. Partial characters of an algebra $A$ are defined in a way such that the corresponding functions $\widehat{x} \colon A_x^{\circledast} \to \Real$ are partial characters of $A$. As separated $C^*$-algebras can be defined in a way analogous to the definition of separated swap algebras, 
a Stone-\v Cech theorem for commutative $C^*$-algebras is highly expected to be provable. Such a theorem would explain why a restriction of the Gelfand representation to separated $C^*$-algebras is no loss of generality from the point of view of the theory of characters of such an algebra. 


An important extension of this work is the constructive proof of a topological Stone representation theorem for separated swap algebras of type $(\tii)$, as a constructive counterpart to the classical topological Stone representation theorem for Boolean algebras. The constructive notion of topological space that seems more suitable to us for such an extension is that of a \textit{complemented topological space of type} $(\tii)$. Complemented topological spaces of type $(\ti)$, introduced in~\cite{Pe24a}, are sets equipped with topologies $\B {\C T}$ of open complemented subsets i.e., sets of complemented subsets $(G^1, G^0)$ that satisfy the usual axioms for a  topology with respect to the operations of union and intersection of type $(\ti)$. If these axioms are satisfied with respect to the operations of union and intersection of type $(\tii)$, then we get the complemented topological spaces of type $(\tii)$. A $\B {\C T}$-closed set is a complemented subset $(F^1, F^0)$, such that its complement $(F^0, F^1)$ is in $\B {\C T}$. In this way the classical duality between open and closed sets is recovered constructively. As the clopen sets of a topological space form a Boolean algebra, the clopen sets of a complemented topological space of type $(\tii)$ form a swap algebra of type $(\tii)$. A notion of \textit{swap space of type} $(\tii)$ needs to be identified i.e, a complemented topological space of type $(\tii)$ such that the Stone duality between Boolean algebras and Boolean spaces is extended to a Stone duality between separated swap algebras of type $(\tii)$ and swap spaces of type $(\tii)$. 
	





%
%

\end{document}